\documentclass[a4paper,12pt]{amsart}
\usepackage[utf8]{inputenc}
\usepackage[english]{babel}
\usepackage{amsfonts}
\usepackage{amsmath}
\usepackage{mathrsfs}
\usepackage{amssymb}
\usepackage{amsthm}
\usepackage{amscd}
\usepackage{latexsym}
\usepackage{amstext}
\usepackage{amsxtra}
\usepackage[alphabetic,backrefs,lite]{amsrefs} % for bibliography
\usepackage{color}

\definecolor{darkgreen}{rgb}{0.0, 0.5, 0.0}

\usepackage[all]{xy}
\usepackage{wasysym}

\usepackage{geometry}
\usepackage{enumerate}
\usepackage{wasysym}
\usepackage{latexsym} % for \Box
\usepackage{colonequals}
\usepackage{ulem}

\usepackage[
        draft=false,
        colorlinks, citecolor=darkgreen,
        backref,       
        pdfauthor={F. F. Favale, D. Bricalli, G.P. Pirola},
        pdftitle={On the irreducibility of Hessian loci of cubic hypersurfaces},
        linktocpage        
]{hyperref}
\hypersetup{citecolor=blue,linktocpage}

\newtheorem{theorem}{Theorem}[section]
\newtheorem{lemma}[theorem]{Lemma}

\newtheorem*{theorem*}{Theorem}
\newtheorem*{THA}{Theorem A}
\newtheorem*{THB}{Theorem B}
\newtheorem*{THC}{Theorem C}
\newtheorem{corollary}[theorem]{Corollary}
\newtheorem{proposition}[theorem]{Proposition}
\newtheorem{conjecture}[theorem]{Conjecture}
\newtheorem{question}[theorem]{Question}
\newtheorem{example}[theorem]{Example}
\newtheorem{framework}[theorem]{Framework}
\theoremstyle{definition}
\newtheorem{definition}[theorem]{Definition}

\theoremstyle{remark}
\newtheorem{remark}[theorem]{Remark}

\numberwithin{equation}{section}

\newcommand{\nc}{\newcommand} 
\nc{\cH}{{\mathcal H}}
\nc{\cA}{{\mathcal A}}
\nc{\cG}{{\mathcal G}}
\nc{\cC}{{\mathcal C}}
\nc{\cD}{{\mathcal D}}
\nc{\cO}{{\mathcal O}}
\nc{\cI}{{\mathcal I}}
\nc{\cB}{{\mathcal B}}
\nc{\cY}{{\mathcal Y}}
\nc{\cK}{{\mathcal K}} 
\nc{\cX}{{\mathcal X}}
\nc{\cS}{{\mathcal S}}
\nc{\cE}{{\mathcal E}}
\nc{\cF}{{\mathcal F}}
\nc{\cZ}{{\mathcal Z}}
\nc{\cQ}{{\mathcal Q}}
\nc{\cN}{{\mathcal N}}
\nc{\cP}{{\mathcal P}}
\nc{\cL}{{\mathcal L}}
\nc{\cM}{{\mathcal M}}
\nc{\cT}{{\mathcal T}}
\nc{\cW}{{\mathcal W}}
\nc{\cR}{{\mathcal R}}
\nc{\cU}{{\mathcal U}}
\nc{\cJ}{{\mathcal J}}
\nc{\cV}{{\mathcal V}}
\nc{\bH}{{\mathbb H}}
\nc{\bA}{{\mathbb A}}
\nc{\bG}{{\mathbb G}}
\nc{\bC}{{\mathbb C}}
\nc{\bO}{{\mathbb O}}
\nc{\bI}{{\mathbb I}}
\nc{\bB}{{\mathbb B}}
\nc{\bY}{{\mathbb Y}}
\nc{\bK}{{\mathbb K}} 
\nc{\bX}{{\mathbb X}}
\nc{\bS}{{\mathbb S}}
\nc{\bE}{{\mathbb E}}
\nc{\bF}{{\mathbb F}}
\nc{\bZ}{{\mathbb Z}}
\nc{\bQ}{{\mathbb Q}}
\nc{\bN}{{\mathbb N}}
\nc{\bP}{{\mathbb P}}
\nc{\bL}{{\mathbb L}}
\nc{\bM}{{\mathbb M}}
\nc{\bT}{{\mathbb T}}
\nc{\bW}{{\mathbb W}}
\nc{\bU}{{\mathbb U}}
\nc{\bD}{{\mathbb D}}
\nc{\bJ}{{\mathbb J}}
\nc{\bV}{{\mathbb V}}
\nc{\bbZ}{{\mathbb Z}}
\nc{\bR}{{\mathbb R}}
\nc{\fr}{{\rightarrow}}
\nc{\co}{{\nabla}}

\nc{\cu}{{\barline{\nabla}}}

\newcommand{\un}[1]{\underline{#1}}

\usepackage{mathtools}

\DeclareMathOperator{\Ima}{Im}

\DeclareMathOperator{\Ann}{Ann}

\DeclareMathOperator{\Rank}{Rank}
\DeclareMathOperator{\Sing}{Sing}

\DeclareMathOperator{\BL}{BL}
\DeclareMathOperator{\Gr}{Gr}
\DeclareMathOperator{\Hom}{Hom}

\DeclareMathOperator{\pr}{pr}

\DeclareMathOperator{\Ve}{Vert}

\newcommand{\pa}[1]{{\partial_{#1}}}

\usepackage{tcolorbox}
\newtcolorbox{mybox}{colback=blue!5!white,colframe=blue!75!black}

\begin{document}

\title[On the irreducibility of Hessian loci of cubic hypersurfaces]{On the irreducibility of Hessian loci of cubic hypersurfaces}

\date{May 1, 2024}

\author{Davide Bricalli}
\address{Dipartimento di Matematica,
	Universit\`a degli Studi di Pavia,
	Via Ferrata, 5
	I-27100 Pavia, Italy}
\email{davide.bricalli@unipv.it}

\author{Filippo Francesco Favale}
\address{Dipartimento di Matematica,
	Universit\`a degli Studi di Pavia,
	Via Ferrata, 5
	I-27100 Pavia, Italy}
\email{filippo.favale@unipv.it}

\author{Gian Pietro Pirola}
\address{Dipartimento di Matematica,
	Universit\`a degli Studi di Pavia,
	Via Ferrata, 5
	I-27100 Pavia, Italy}
\email{gianpietro.pirola@unipv.it}

% \thanks{}

\date{\today}
\thanks{
\\
\noindent {\bf Acknowledgements}: \\
The authors wants to express their gratitude to Carlos D'Andrea and Giorgio Ottaviani for stimulating discussions and for pointing out some interesting papers related to some of the topics here treated. 
% The authors also thank the anonymous referees for their useful observations. 
The authors are partially supported by INdAM-GNSAGA and by PRIN 2022 ``\emph{Moduli spaces and special varieties}''. The first and second authors are partially supported by the INdAM – GNSAGA Project, ``\emph{Classification Problems in Algebraic Geometry: Lefschetz Properties and Moduli Spaces}'' (CUP$\_$E55F22000270001)
}

%    General info
\subjclass[2020]{Primary: 14J70; Secondary: 14M12, 14J17, 14J30, 14J35, 14C34}
% 14C34: Torelli problems
% 14M12: Determinantal varieties
% 14J70: Hypersurfaces and algebraic geometry
% 14J17: Singularities of surfaces or higher dimensional varieties
% 13E10: Commutative artinian rings and modules, finite-dimensional algebras
% 14J30: 3-folds
% 14J35: 4-folds

%    General info
%\subjclass[2020]{Primary: 13E10; Secondary: 14J70, 14N05, 14M10, 13F20}

\keywords{Hessian varieties, cubic hypersurfaces, Thom-Sebastiani, irreducibility}

% ---------------------------------------------------------
% ---------------------------------------------------------
% ---------------------------------------------------------

\maketitle
\begin{abstract}
    We study the problem of the irreducibility of the Hessian variety $\cH_f$ associated with a smooth cubic hypersurface $V(f)\subset \bP^n$. We prove that when $n\leq5$, $\cH_f$ is normal and irreducible if and only if $f$ is not of Thom-Sebastiani type, i.e.,
    roughly, one can not separate its variables. This also generalizes a result of Beniamino Segre dealing with the case of cubic surfaces. The geometric approach is based on the study of the singular locus of the Hessian variety and on infinitesimal computations arising from a particular description of these singularities.
\end{abstract}
% ---------------------------------------------------------

\section*{Introduction}

Let $X=V(f)$ be a hypersurface of the projective space $\bP^n$ over an algebraically closed field $\bK$ of characteristic zero. In the case where the determinant $h_f=\det(H_f)$ of the associated Hessian matrix $H_f$ is not equivalently zero, for example for $X$ smooth, it is well-known that the associated Hessian hypersurface $\cH_f=V(h_f)\subset \bP^n$ contains many information of $X$ itself.\\ A sort of {\it {generic Torelli theorem}} for Hessian hypersurfaces is also supposed to be valid (see \cite{CO}), up to some known cases. In particular, in \cite{CO} it is studied the so-called {\it {Hessian map}} 
$$h_{d,n}:\bP(S^d)\dashrightarrow\bP(S^{(n+1)(d-2)}) \qquad [f]\mapsto [h_f]$$
where $S$ denotes the ring $\bK[x_0,\dots,x_n]=\oplus_{d\geq0} S^d$. In the specific case of cubic hypersurfaces, namely when $d=3$, it is known that $h_{3,1}$ has the generic fiber of dimension $1$, $h_{3,2}$ has the generic fiber consisting of $3$ points, $h_{3,3}$ is birational onto its image. As in this last case, the {\it Ciliberto-Ottaviani conjecture} states that the hessian map $h_{d,n}$ should be birational onto its image, for higher values of $n$ and even for different values of $d$. Very recently, this conjecture has been dealt with for the case of curves (of any degree) in \cite{Beo} and \cite{COCD}.
\smallskip

One can argue that also the singular locus $\Sing(\cH_f)$ of the Hessian hypersurface $\cH_f$ must keep track of some crucial aspects of $X$. Along this line, the aim of this paper is a deep study of the singularities of the hypersurface $\cH_f$ associated with a smooth cubic $(n-1)$-fold; in particular, we are interested in the dimension of $\Sing(\cH_f)$ and in the irreducibility of $\cH_f$ itself. Notice that, for low dimensional hypersurfaces, i.e. curves or surfaces, this analysis is classical (see, for instance, \cite{Hut} and \cite{Seg}). More recently, an approach has been developed in \cite{AR} for the case of cubic threefolds, while in \cite{BFP2} the authors have dealt with the higher dimensional cases.
\smallskip

To explain our main result (Theorem $A$), let us recall that a polynomial $f$ (or a hypersurface $X=V(f)$) is said to be {\it of Thom-Sebastiani type}, TS for brevity, if up to a change of coordinates it can be written using two distinct sets of variables (see Definition \ref{DEF:TS}). The name comes from the works \cite{TS} of Sebastiani and Thom and \cite{Tho} of Thom. These polynomials have been extensively studied in several contexts (for example, about their Jacobian ideals in some classical works of Bertini, Longo and Mammana - \cite{Ber}, \cite{Lon} and \cite{Mamm}) and have appeared with other names too in the literature (for example they are called {\it direct sums} in \cite{BBKT} and \cite{Fed}). The Hessian hypersurface associated with a TS polynomial is not irreducible and its singular locus has dimension $n-2$. 
%Hence, this ``hessian test'', namely the computation of the dimension of the singular locus of the Hessian variety, can prove that a polynomial is not of TS type. 
The interesting fact is that this is actually a characterization, as proved in:
\smallskip

\begin{THA}[Theorem \ref{THM:star}]
Assume that $n\leq 5$ and consider $f\in \bK[x_0,\dots,x_n]$ defining a smooth cubic. Then the Hessian hypersurface $\cH_f\subset\bP^n$ is irreducible and normal if and only if $f$ is not of Thom-Sebastiani type.    
\end{THA}
\smallskip
The problem of determining whether a polynomial is of TS type is interesting and investigated in the literature (see, for example, \cite{BBKT} or \cite{Fed}), also from an algorithmic point of view. As said above, one can apply a sort of ``hessian test'': if the Hessian hypersurface associated to a polynomial $f$ is normal, then $f$ can't be of TS type. Furthermore, with a strong geometric approach, Theorem $A$ guarantees that this ``hessian test'' is actually a complete test for (smooth) cubic forms with at most 6 variables. 
\smallskip

Beyond the smoothness hypothesis, which is anyway necessary (see Remark \ref{REM:counterexamples} for details), one could conjecture that the same result still holds for higher dimensions or higher degrees. Even if we strongly believe that Theorem $A$ is valid for smooth cubic hypersurface of any dimension, one can see that this is not the case, for example, in degree $4$ (see again Remark \ref{REM:counterexamples}).
\smallskip

Besides the fact that cubic hypersurfaces are classically endowed with interesting and particular properties in relation to their geometry and also, for example, to their associated Hessian variety (see, for example, \cite{Dol}, \cite{Rus}, \cite{GR}, \cite{Huy},...), the peculiarity of the case of (smooth) cubics lies in the framework we want to deal with and in the techniques used. Indeed, Hessian loci of cubic hypersurfaces are equipped, among other things, with a special symmetry that will be a key ingredient in the whole article and which makes a crucial difference with the higher degree cases. Indeed, if $X=V(f)$ is a general cubic hypersurface then its hessian $\cH_f$ is a singular Calabi-Yau variety with a fixed point free rational involution. Indeed, for $d=3$, one can observe that 
$$H_f(x)\cdot y=H_f(y)\cdot x.$$
Such a relation can be easily translated in terms of the associated apolar ring $A_f=D/\Ann_D(f)$, where $D=\bK[\frac{\partial}{\partial x_0},\dots,\frac{\partial}{\partial x_0}]$, as $x\cdot y=y\cdot x$. Under the natural identification $\bP(A^1)\simeq \bP^n$, (which comes from the Gorenstein duality of the apolar ring - see \cite{Mac94} or the comprehensive book \cite{bookLef}), one can define the natural incidence correspondence
$$\Gamma_f=\{([x],[y])\in\bP^n\times\bP^n \ | \ H_f(x)\cdot y=0\}=\{([x],[y])\in\bP(A^1)\times\bP(A^1) \ | \ x\cdot y=0\}.$$
Observe that it dominates, via the two projections, the Hessian hypersurfaces $\cH_f$. Moreover, $\Gamma_f$ is also equipped with an involution $\tau(([x],[y]))=([y],[x])$ which is fixed point free and descends to the above-mentioned fixed point free rational involution defined on the Hessian hypersurface $\cH_f$.
\smallskip

Another key point is the determinantal structure of the Hessian hypersurface. Indeed, $\cH_f$ is equipped with a natural rank-decreasing filtration
$$\cH_f=\cD_n(f)\supseteq \mbox{Sing}(\cH_f)\supseteq\cD_{n-1}(f)\supseteq\dots \supseteq D_i(f)\supseteq\dots  \supseteq D_{0}(f),$$ 
where $\cD_k(f)=\{[x]\in\bP^n \ | \ \mbox{Rank}(H_f(x))\leq k\}$. In the case where $V(f)$ is {\it any} smooth cubic hypersurface, the second inclusion in the above filtration is actually an equality as proved in \cite{AR} for cubic threefolds and in \cite{BFP2} for $n\geq5$. Moreover, the authors proved that for $X$ general, each of the $D_i$'s has the expected dimension and we computed the basic invariant in some low dimensional cases. In particular, for a smooth cubic $V(f)$, we have 
$$\mbox{dim}(\mbox{Sing}(\cH_f))\in\{n-3,n-2\}.$$
Indeed, $\cH_f$ is reduced and the expected value, namely $n-3$, is achieved for $f$ general. Observe that in the other case, we would have that the Hessian locus $\cH_f$ is not normal (indeed, see \cite[Proposition 8.23]{Hart}, recall that for a hypersurface in $\bP^n$ being normal is equivalent to be regular in codimension $1$). Then, with Theorem $A$, we partially answer the following natural question:
\begin{center}
{\it
    When, for a smooth cubic hypersurface $V(f)$, is $\cH_f$ not normal?
    }
\end{center}
When $n=2$, given a smooth cubic curve $X=V(f)$, the associated Hessian curve $\cH_f$ is singular if and only if $X$ is the Fermat curve. It is remarkable that Beniamino Segre in 1943 proved that a similar result also holds for cubic surfaces in a projective $3$-dimensional space. Indeed, (see \cite{Seg}) given a smooth cubic surface $X=V(f)\subset\bP^3$, $\cH_f$ is reducible if and only if $X$ is cyclic, that is up to a change of coordinate we can write $f=z_0^3+g(z_1,z_2,z_3)$.
\smallskip

It is a real misfortune, due likely to the war and to racial issues, that the book of Beniamino Segre which focuses on non-singular cubic surfaces, is not easy to be found. Its analysis is based on the use of {\it Sylvester's Pentahedral Theorem} which, with a modern terminology, says that the general cubic surface has {\it Waring rank} equal to $5$, i.e., after a suitable change of coordinates, it can be written as the zero locus of $\sum_{i=1}^{5}L_i^3$ where $L_1,\dots,L_5\in S^1$. This description has also been recently used for example in \cite{DVG} and \cite{CO}. With Theorem $A$, we extend  Segre's result to cubic threefolds and fourfolds. 
\medskip

Let us now explain the main features of our proof. First of all, let us observe that, since no useful ``Sylvester form tool'' seems to exist for cubic forms in $\bP^n$ with $n\geq4$, a completely new strategy must be used. In this environment, it is a great pleasure to acknowledge our main source of inspiration: Adler's work. In a remarkable series of appendices to the book \cite{AR}, among many other results, Adler set up a method to study the singular locus of the Hessian locus $\cH_f$ associated with a cubic hypersurface $V(f)$. He considered the correspondence $\Gamma_f$ introduced above, which can be seen as a partial desingularization of $\cH_f$ and moreover, he proved that the singular locus of $\Gamma_f$ has a ``triangle structure''. More precisely, a point $([x],[y])$ is singular for $\Gamma_f$ if and only if there exists $[z]\in\bP^n$ such that $([x],[y]),([x],[z]),([y],[z])\in\Gamma_f$.
\smallskip

Our crucial observation is that if $\mbox{Sing}(\cH_f)$ contains a component of dimension $n-2$ then the same holds also for $\mbox{Sing}(\Gamma_f)$. We have then a large amount of triangles to deal with and, moreover, such a description is greatly enlightened by using the {\it {apolar-geometric}} method we already exploited in our proof of a Gordan-Noether theorem (see \cite{GN}, \cite{Rus}, \cite{BFP}). The whole proof is then devoted to showing that ``too many triangles'' for $\cH_f$ force $f$ to be of TS type. 
\smallskip

Two results can be thought of as the main ingredients to this aim. First of all, we give a characterisation of the cubic polynomials of TS type in terms of the Hessian loci $\cD_k$ appearing in the above-mentioned filtration:

\begin{THB}[Theorem \ref{THM:charactTS}]
A polynomial $f\in \bK[x_0,\dots,x_n]$ defining a smooth cubic is of TS type of the form $f(x_0,\dots,x_n)=f_1(x_0,\dots,x_k)+f_2(x_{k+1},\dots,x_n)$ if and only if $\cD_{k+1}(f)$ contains a $\bP^{k}$.
\end{THB}

The second result allows us to make specific assumptions on the general triangle we will deal with. In particular, by considering an irreducible component $\cF$ of the variety parametrizing these triangles for $\cH_f$ and denoting by $\pi_i$ its natural projections, we have 

\begin{THC}[Theorem \ref{THM:moon}]
Assume $n\leq 5$ and let $X=V(f)$ be a smooth cubic hypersurface in $\bP^n$ not of TS type. If $\cF$ is an irreducible family of triangles for $\cH_f$ with $\dim(\pi_1(\cF))=\dim(\cF)=n-2$, then the general element in $\cF$ is such that none of its vertices belongs to $X$.
\end{THC}

Even if, a fortiori, the situation presented in the above theorem can not be realised, let us stress that this result will allow us to set the right framework on which all the proof is based.
\smallskip

The problem (after some reduction preliminaries) becomes to compute the Zariski tangent space at the general point of $\cF$. 
This approach leads almost immediately to a conclusion in the case of cubic surfaces and it is reasonably accessible for $n=4$.
In the fourfold case, the computation becomes instead much more complicated: there are really many sub-cases to be considered (this is certainly due also to the fact that for $n\leq 4$ all the cubic of Thom-Sebastiani type are indeed cyclic, which is not true anymore for $n\geq5$). 
\smallskip

It is interesting to notice that there is a family of cubic fourfolds (which is considered in Lemma \ref{LEM:diagonalizable:VeryBadCase}), where the infinitesimal methods are not enough in order to conclude. For these hypersurfaces, in the spirit of the possible {\it {Torelli theorem}}, we recover the equation of the cubic fourfold $V(f)$ and then, with a direct computation we show that the dimension of the singular locus of $\cH_f$ is actually the expected one, i.e. $2$, unless $f$ is of TS type.
\smallskip

% In particular, in a very heavily convoluted way, one of these cases, despite to the big amount of relations we are able to find, appears to be too involved to be excluded directly only by infinitesimal methods. 
% For such a case, by using these relations, we recover the equation of the cubic fourfold (in the spirit of the possible {\it {Torelli theorem}}) and, finally, a direct computation shows that the dimension of the singular locus of the associated Hessian locus is actually $2$. 

The plan of the article is the following. After setting the notation and proving some preliminary results in Section \ref{SEC:1}, in Section \ref{SEC:TS} we deal with polynomials of Thom-Sebastiani type and we prove Theorem $B$. In Section \ref{SEC:FamiliesOfTriangles}, we focus on the study of particular families of triangles and we prove Theorem $C$. Finally, in Sections \ref{SEC:star_CubicThreefoldCase} and \ref{SEC:star_CubicFourfoldCase} we prove our main result, namely Theorem $A$, respectively for $n\leq4$ and for the case of cubic fourfolds.

% ---------------------------------------------------------
% ---------------------------------------------------------
% ---------------------------------------------------------

\section{Preliminaries and first results}
\label{SEC:1}

In this first section, we set the notation and present some preliminary results, some of them proved in \cite{BFP2}. For a complete comprehension of standard Artinian Gorenstein Algebras, which we are going to introduce, one can refer to \cite{bookLef}. Consider $\bK$ an algebraically closed field of characteristic $0$ and the projective space $\bP^n$ for $n\geq 2$. Let us set 
$$S=\bK[x_0,\dots,x_n]=\bigoplus_{k\geq0}S^k\qquad \mbox{ and }\qquad \cD=\bK[y_0,\dots,y_n]=\bigoplus_{k\geq0}\cD^k$$
so that $S$ is the homogeneous coordinate ring of $\bP^n$ and $\cD$ is the graded algebra of linear differential operators on $S$, where we define $y_i$ as the first partial derivative $\frac{\partial}{\partial x_i}$. If $v\in \bK^{n+1}$, we will denote by $\pa{v}$ the derivative in the direction of $v$, i.e. $\sum_{i=0}^n v_iy_i$.
\smallskip

Let us now consider a homogeneous polynomial $f$ of degree $d$, i.e. an element of $S^d$. Two objects can then be associated with $f$ in a natural way:
\begin{itemize}
    \item the {\it Jacobian ring} of $f$, defined as the quotient $R_f=S/J_f$, where $J_f$ denotes the Jacobian ideal of $f$, spanned by the partial derivatives of $f$;
    \item the {\it apolar ring} of $f$, defined as the quotient $A_f=\cD/\mathrm{Ann}_{\cD}(f)$, where $\mathrm{Ann}_\cD(f)$ is the annihilator ideal of $f$, i.e. the ideal in $\cD$ given by $\{\delta\in \cD \ | \ \delta(f)=0\}$.
\end{itemize}

Both the Jacobian and the apolar ring of $f$ are graded Artinian algebras with socle in degree respectively $(n+1)(d-2)$ and $d$, i.e. $R_f=R^0\oplus R^1\oplus\cdots\oplus R^{(n+1)(d-2)}$ and $A_f=A^0\oplus A^1\oplus\cdots\oplus A^d$. One can also see that they are standard, i.e. generated in degree $1$, and that they satisfy the so-called Poincar\'e (or Gorenstein duality), i.e. for example the multiplication map $A^{d-k}\times A^k\rightarrow A^d$ is a perfect pairing for every suitable positive integer $k$. In other words, they are both examples of what we call SAGAs, an acronym for standard Artinian Gorenstein algebras. 
%introduced in \cite{BFP} and \cite{BF}, where the authors started the study of a specific framework, here developed.
\smallskip

Finally, given $f$ as above, we can then define the associated {\it Hessian matrix} and the {\it hessian polynomial}, respectively the square symmetric matrix whose entries are the second partial derivatives of $f$ with respect to the $x_i$'s and the determinant of such a matrix, i.e.
$$H_f=((y_iy_j)(f))_{i,j=0,\dots,n} \qquad \mbox{and} \qquad h_f=\det(H_f).$$

Let us observe that if the zero locus of $f$, $X=V(f)\subset\bP^n$, is a smooth hypersurface, then the hessian determinant $h_f$ belongs to $S^{(n+1)(d-2)}\setminus\{0\}$. In this case, one can define the {\it Hessian hypersurface} $\cH_f$ associated with $f$ (or with $X$) as the zero locus of such a polynomial, i.e. 
$$\cH_f=V(h_f).$$
The smoothness of $X=V(f)$ implies also that the associated apolar ring $A_f$ is such that $A^1$ has dimension $n+1$: indeed such a dimension is strictly smaller than $n+1$ if and only if $V(f)$ is a cone. From the natural pairing $S\times \cD\rightarrow S$, one can then deduce an isomorphism
$$\bP^n\simeq \bP((S^1)^*)\simeq \bP(A^1).$$

From now on let us focus on the case {\bf{$d=3$}}: the first result we need to recall is the following Proposition (see \cite[Proposition 1.2]{BFP2}), which allows us to interpret the cubic hypersurface $X=V(f)$, its singular locus and its associated Hessian variety in terms of the apolar ring $A_f$.

\begin{proposition}
\label{PROP:descrapol}
Given a cubic hypersurface $X=V(f)$ (not a cone) and the corresponding $A_f$, we have 
\begin{enumerate}[(a)]
    \item Under the identification $\bP^n\simeq \bP(A^1)$, computing $H_f(x)\cdot y$, for $[x],[y]\in \bP^n$, is equivalent to compute $xy\in A^2$;
    \item $X=\{[y]\in \bP(A^1)\,|\, y^3=0\}$;
    \item $\Sing(X)=\{[y]\in \bP(A^1)\,|\, y^2=0\}$;
    \item $\cH_f=\{[y]\in \bP(A^1)\,|\,\, \exists\, [x]\in \bP(A^1) \mbox{ with } xy=0\}$.
\end{enumerate}
\end{proposition}

In this paper, we deal with a homogeneous cubic polynomials whose zero locus is smooth: we will denote by $\cU\subset\bP(S^3)$, the locus of such elements. 
\smallskip

As done in \cite{BFP2}, given $[f]\in\cU$, let us introduce some objects which will be used extensively in what follows. First of all, for $[x]\in \bP^n$ we set
\begin{equation}
\iota([x])=\bP(\ker(H_f(x))).
\end{equation}
This is either empty (exactly when $[x]\not\in \cH_f$) or a projective linear space of dimension $n-\Rank(H_f(x))$. It is then natural to consider the {\it Hessian loci}
\begin{equation}
\cD_k(f)=\{[x]\in\bP^n \ | \ \Rank(H_f(x))\leq k\}
\end{equation}
which give a stratification of the projective space $\bP^n$ and in particular of the Hessian hypersurface $\cH_f$ (for example, we have $\cD_{n+1}(f)=\bP^n$ and $\cD_n(f)=\cH_f$). Moreover, in general, for $k\leq n-1$, $\cD_{k-1}(f)\subseteq\cD_k(f)\subset\cH_f$. We will simply write $\cD_k$, when it is clear which polynomial we are referring to in the following. In \cite{BFP2} the authors actually proved that for every suitable $k$, $\cD_{k-1}(f)\subseteq\Sing(\cD_k(f))$ and that equality holds for $[f]\in \cU$ general.
For any $[f]\in \cU$, let us introduce a useful incidence correspondence:
\begin{equation}
    \Gamma_f=\{([x],[y])\in\bP^n\times\bP^n \ | \ H_f(x)\cdot y=0\}
\end{equation}
and let us denote by $\pr_i$ the two natural projections

\begin{remark}
\label{REM:Tau}
By the relation $H_f(x)\cdot y=H_f(y)\cdot x$ (which is equivalent to the relation $xy=yx$ in $A_f$, by Proposition \ref{PROP:descrapol}), the standard involution
$$\tau:\bP^n\times \bP^n\to \bP^n\times \bP^n\qquad ([x],[y])\mapsto ([y],[x])$$
sends $\Gamma_f$ to itself. As a consequence, $\Gamma_f$ dominates $\cH_f$ via both $\pr_1$ and $\pr_2$.\\
\end{remark}

By Proposition \ref{PROP:descrapol}, the loci just introduced can be described also in terms of the apolar ring as follows:
$$
\iota([x])=\bP(\ker(x\cdot :A^1\to A^2))\qquad
\Gamma_f=\{([x],[y])\in\bP(A^1)\times\bP(A^1) \ | \ xy=0\}.
$$
Through the article, we will use one description or the other according to the convenience.
\smallskip

We summarise here the main results of \cite{BFP2}:

\begin{theorem}
\label{THM:fundamental}
The following hold:
\begin{enumerate}[(a)]
    \item for {\it any} $[f]\in \cU$, we have $\Sing(\cH_f)=\cD_{n-1}(f)$;
    \item if $[f]\in \cU$ is general, then $\Gamma_f$ is smooth and $\pr_i:\Gamma_f\rightarrow\cH_f$ is a desingularization;
    \item the expected codimension of $\cD_k(f)$ is $\binom{n-k+2}{2}$.
\end{enumerate}
In particular, the expected dimension of $\Sing(\cH_f)$ equals $n-3$.
\end{theorem}

Hence, by Theorem \ref{THM:fundamental}, the Hessian variety associated with any smooth cubic hypersurface in $\bP^n$ for $n\geq3$ is always singular and, in the general case, $\Gamma_f$ is a desingularization for it and we have a lower bound for the dimension of $\Sing(\cH_f)$ for {\it any} $[f]\in \cU$. We are now interested in giving an upper bound for this dimension and more generally for $\dim(\cD_k(f))$.

\begin{remark}
\label{REM:BIGNESSofDELTA}
It is well-known that the diagonal $\Delta_{\bP^n}\subseteq \bP^n\times \bP^n$ has decomposition in the Chow group of $\bP^n\times \bP^n$ given by $$[\Delta_{\bP^n}]=\bigoplus_{p+q=n}pr_1^*[H]^p\cdot pr_2^*[H]^q,$$ 
where $H$ is a hyperplane in $\bP^n$ and $\pi_i$ are the standard projections. Hence, every effective cycle of dimension at least $n$ intersects $\Delta_{\bP^n}$. 
\end{remark}

\begin{proposition}
\label{PROP:Comp_Of_Gamma}
Consider {\it any} $[f]\in \cU$. Then the following hold:
\begin{enumerate}[(a)]
    \item the variety $\Gamma_f$ is a connected complete intersection in $\bP^n\times \bP^n$ of pure dimension $n-1$;
    \item for each $k$, one has $\dim(\cD_k(f))\leq k-1$;
    \item there is a bijective correspondence between irreducible components of $\Gamma_f$ and the irreducible components of the various loci $\cD_{k}(f)$ for which the bound in $(b)$ is sharp.

\end{enumerate}
\end{proposition}

\begin{proof}
For $(a)$, first of all observe that $\Gamma_f\cap \Delta_{\bP^n}=\emptyset$ since, otherwise, $V(f)$ would be singular by Proposition \ref{PROP:descrapol}. Hence, by Remark \ref{REM:BIGNESSofDELTA}, we have that $\Gamma_f$ has dimension at most $n-1$. On the other hand, by definition, we have that $\Gamma_f$ is cut by $n+1$ divisors of $\bP^n\times \bP^n$ of bidegree $(1,1)$ so each component of $\Gamma_f$ has dimension at least $n-1$. \\
Since $\Gamma_f$ is a complete intersection, its connectedness follows by the Fulton-Hansen-type theorem (see \cite{FH}, \cite[Ch.3]{Laz} or \cite{MNP}).

\smallskip

For $(b)$, let us assume by contradiction that there exists an irreducible component $W$ of $\cD_k(f)$ of dimension $d\geq k$. Over the general point $[w]$ of $W$, the fiber of the projection $\pr_1$ from $\Gamma_f$ is a projective space $\iota([x])\simeq \bP^s$ with $s\geq n-k$. Therefore there exists a component of $\Gamma_f$ of dimension at least $d+s\geq n$. This would mean, by Remark \ref{REM:BIGNESSofDELTA}, that $\Gamma_f\cap\Delta_{\bP^n}$ is not empty, giving a contradiction. 
\medskip

For $(c)$, assume that $W$ is as in $(b)$ and of dimension $k-1$. Then the same reasoning as above implies the existence of an irreducible component of $\Gamma_f$ dominating $W$. For the converse, let $G$ be an irreducible component of $\Gamma_f$, set $G'=\pr_1(G)$ and let $m$ be the dimension of $G'$.  If $m=n-1$, then $G'$ is a component of $\cH_f=\cD_n(f)$ so we are done. Otherwise, $m=n-1-a$ with $a>0$ so that $\pr_1|_G$ has the general fiber $F$ of dimension $a$. Since the whole fiber of $\pr_1$ over a general point of $[x]\in G'$ is a projective space containing a fiber of dimension $a$, we have that $\Rank(H_f(x))\leq n-a$, i.e. the general point of $G'$ lies in $\cD_{n-a}(f)$, as claimed.
\end{proof}

As a consequence of Proposition \ref{PROP:Comp_Of_Gamma}, recalling that $\Sing(\cH_f)=\cD_{n-1}(f)$, we have in particular that 
$$\dim(\Sing(\cH_f))\in\{n-3,n-2\}$$
and for $f$ general such dimension coincides with the expected one, i.e. $\dim(\Sing(\cH_f))=n-3$. In this paper we are interested in answering the following:
\begin{question}
\label{QST:question}
    For which $[f]\in\cU$ does $\Sing(\cH_f)$ have dimension $n-2$?
\end{question}

In this study, the singularities of $\Gamma_f$ will play a central role.

\begin{definition}
Given $f\in S^3$, we set 
$$\cT=\{([x],[y],[z])\in (\bP^n)^3\,|\, [x]\in\iota(y), \ [y]\in\iota(z) \mbox{ and } [z]\in\iota(x)\}.$$
Elements in $\cT$ are called {\it triangles for $\cH_f$}.
\end{definition}

\begin{remark}
\label{REM:CoseScemeSuiTriangoli}
Recall that if $f$ is a cubic polynomial, then the incidence variety $\Gamma_f$ is symmetric with respect to the involution $\tau$. This implies that $$[x]\in\iota(y) \Longleftrightarrow [y]\in\iota(x)$$
for all $[x],[y]\in \bP^n$ so a permutation of the {\it vertices} of a triangle yields again a triangle. Furthermore, if $[f]\in \cU$, two vertices of the same triangle cannot be equal, since we have $\Delta_{\bP^n}\cap \Gamma_f=\emptyset$, and so by construction, each vertex of a triangle lies necessarily in $\cD_{n-1}(f)$. 
\end{remark}

The following result links triangles for $\cH_f$ and singularities of $\Gamma_f$. It has been proved for $n=4$ in \cite{AR} and in \cite{BFP2} for the general case.
\begin{lemma}
    \label{LEM:Triangles_and_Singularities}
    For $[f]\in\bP(S^3)$, a point $([x],[y])$ is singular for $\Gamma_f$ if and only if there exists a third point $[z]\in\cH_f$ such that the triple $([x],[y],[z])$ is a triangle for $\cH_f$.
\end{lemma}

To conclude this first section, let us present a couple of technical results which will be useful in what follows.

\begin{lemma}
\label{LEM:squareindep}
Assume that $[f]\in \cU$. Consider a triangle $T\in\cT$ and a point $P\in \Gamma_f$. Then
\begin{enumerate}[(a)]
    \item the squares of the coordinates of $P$ are independent;
    \item the vertices of $T$ span a $\bP^2$;
    \item the squares of the vertices of $T$ are independent.
\end{enumerate}
\end{lemma}

\begin{proof}
In order to prove the claims, we will use extensively that $\Sing(V(f))$ can be identified with $\{[x]\in \bP(A^1)\,|\, x^2=0\}$ (see Proposition \ref{PROP:descrapol}). More precisely, we will proceed by contradiction by proving that if the conclusion of $(a),(b)$ or $(c)$ are false, then there exists an element whose square is $0$, i.e. a singular point for $f$, which is impossible by assumption.

Let us start by proving $(a)$. For $P=([y_1],[y_2])\in\Gamma_f$, assume, by contradiction, that $y_1^2$ and $y_2^2$ are linearly dependent. Then, there exists $\lambda\in\bK$ such that $y_1^2=-\lambda^2y_2^2$. As $y_1y_2=0$ vanishes, we would have that $(y_1+\lambda y_2)^2=0$, which is impossible.

For $(b)$, first of all, notice that given any pair of vertices of $T=([x],[y],[z])$, they are independent: this follows either by Remark \ref{REM:CoseScemeSuiTriangoli} or by $(a)$. If by contradiction we could write $z=\alpha x+\beta y$ for some $\alpha,\beta\in\bK^*$, then, by multiplying by $x$ we would obtain $0=xz=\alpha x^2$. This would imply $x^2=0$, which is not possible.

For $(c)$, we have then simply to prove that $z^2\not\in\left\langle x^2,y^2\right\rangle$ in $A^2_f$: let us assume again by contradiction that it is the case, i.e. there exist $\alpha$ and $\beta$ in $\bK^*$ such that $z^2=-\alpha^2 x^2-\beta^2 y^2$. In the same way as before, we can consider the square $(z+\alpha x+\beta y)^2$, which is zero since $xy=xz=yz=0$, leading a contradiction.
\end{proof}

Assume that $[f]\in \cU$. If $\cF$ is variety of $(\bP^n)^3\simeq \bP(A^1)^3$ whose points are triangles for $\cH_f$ (i.e. $\cF\subseteq \cT$), we will refer to $\cF$ as a {\it family of triangles (for $\cH_f$)}. Recall that the tangent space to $(\bP^n)^3\simeq \bP(A^1)^3$ at $T=([x_1],[x_2],[x_3])$ is given by
\begin{equation}
T_{(\bP(A^1))^3,T}=\bigoplus_{i=1}^3 A^1/\langle x_i\rangle.
\end{equation}

\begin{lemma}
\label{LEM:TGVec_Triang_First_Order}
Assume that $[f]\in \cU$. Let $T=([x_1],[x_2],[x_3])$ be a triangle for $\cH_f$ and consider a (Zariski) tangent vector $\underline{v}=(x_1',x_2',x_3')\in T_{\cT,T}$. Then one has 
$$x_ix_j'+x_jx_i'=0\qquad  \mbox{ and }\qquad x_i'\in \Ann_{A^1}(x_j^2,x_k^2)/\langle x_i\rangle$$
whenever $\{i,j,k\}=\{1,2,3\}$.
\end{lemma}

\begin{proof}
The variety $\cT$ can be described in $\bP(A^1)^3$ as
$$\{([x],[y],[z])\in (\bP(A^1))^3\,|\, xy=xz=yz=0\}$$
so that the forms $xy,xz$ and $yz$ vanish identically on $\cT$. If $\underline{v}=(x_1',x_2',x_3')\in T_{\cT,T}$ is a tangent vector, then $(x_1+tx_1',x_2+tx_2',x_2+tx_3')$ satisfies the equations $xy=xz=yz=0$ at first order:
$$0=(x_i+tx_i')(x_j+tx_j') \mod t^2=x_ix_j+t(x_ix_j'+x_jx_i') \mod t^2$$
so $x_ix_j'+x_jx_i'=0$ as claimed. Multiplying by $x_j$ we get $x_j^2x_i'=0$, i.e. we have
$x_i'\in\mathrm{Ann}_{A^1}(x_j^2,x_k^2)/\left\langle x_i\right\rangle$.
\end{proof}

As a consequence of Lemma \ref{LEM:squareindep} we can then define a morphism
$$\psi:\cT\rightarrow \mathrm{Gr}(2,\bP^n) \qquad T=([x],[y],[z])\mapsto \bP(\left\langle x,y,z\right\rangle).$$

\begin{proposition}
\label{PROP:injtriangles}
The morphism $\psi$ has everywhere injective differential and it is injective modulo permutations of the vertices.
\end{proposition}

\begin{proof}
Let us start by proving that the map $\psi$ is injective (up to the permutation of such vertices). Let $T=([x_1],[x_2],[x_3])$ and $T'=([y_1],[y_2],[y_3])$ be two triangles which are not equivalent via permutation of the vertices. Assume, by contradiction, that $\psi(T)=\psi(T')$, i.e. the two triples of vertices span the same projective plane. We can then write $y_i=\sum_{j=1}^3a_{ij}x_j$. Recall that $x_kx_l=\delta_{kl}x_k^2$, since $x_kx_l=0$ for every $k\neq l$, while $x_i^2\neq 0, y_i^2\neq 0$ for every $i$, by the smoothness of $V(f)$. Hence, for $i\neq j$, we get
$$0=y_iy_j=\sum_ka_{ik}x_k\cdot\sum_la_{jl}x_l=\sum_{k,l}a_{ik}a_{jl}x_kx_l=a_{i1}a_{j1}x_1^2+a_{i2}a_{j2}x_2^2+a_{i3}a_{j3}x_3^2.$$
   
From Lemma \ref{LEM:squareindep}, we have then that $a_{ik}a_{jk}=0$ for every $k=1,2,3$ and $i\neq j$. From this one can easily see that, for every $i=1,2,3$, at least (and at most, by construction) two coefficients among $a_{i1},a_{i2},a_{i3}$ are trivial. Hence, the vertices of $T'$ and of $T$ are the same up to a permutation. 
\smallskip

Fix a triangle $T=([x_1],[x_2],[x_3])$ in $\cT$. We claim now that the differential 
$$d_T\psi:T_{\cT,T}\rightarrow T_{\mathrm{Gr}(2,\bP(A^1_f)), \left\langle x_1,x_2,x_3\right\rangle}$$
of $\psi$ at $T$ is injective. Let us consider a non-trivial vector 
$$\underline{v}=(x_1',x_2',x_3')\in T_{\cT,T}\subset T_{\bP(A^1)^3,T}\simeq \bigoplus_{i=1}^3 A^1/\left\langle x_i\right\rangle.$$ 
Since $T_{\Gr(k,V),W}\simeq \Hom(W,V/W)$ we have that 
$(d_T\psi)(\underline{v})\in\Hom\left(\left\langle x_1,x_2,x_3\right\rangle, A^1/\left\langle x_1,x_2,x_3\right\rangle\right)$
and, if we assume that $d_T\psi(\underline{v})\equiv0$, we have $x_i'\in\left\langle x_1,x_2,x_3\right\rangle$ so we can write 
\begin{equation}
\label{EQ:Temp_LinDep}
x_i'=\sum_{m=1}^3a_{im}x_m
\end{equation}
for suitable $a_{im}\in \bK$. By Lemma \ref{LEM:TGVec_Triang_First_Order} we have
$x_ix_j'+x_i'x_j=0$ for $i\neq j$ so, using the relations in Equation \eqref{EQ:Temp_LinDep}, we obtain
$$0=x_i\sum_{m=1}^3a_{jm}x_m+x_j\sum_{m=1}^3a_{im}x_m=a_{ji}x_i^2+a_{ij}x_j^2 \qquad \mbox{ for } i\neq j.$$
By Lemma \ref{LEM:squareindep} squares of vertices of a triangle are independent so we obtain $a_{ij}=0$ for $i\neq j$ and $x'_i=a_{ii}x_i$. This is impossible since $x_i'\in A^1/\langle x_i\rangle$ and we would have $\underline{v}=0$ whereas $\underline{v}$ is assumed to be non-trivial.

\end{proof}

% ---------------------------------------------------------
% ---------------------------------------------------------
% ---------------------------------------------------------
% ---------------------------------------------------------
% ---------------------------------------------------------
% ---------------------------------------------------------
% ---------------------------------------------------------

\section{Characterisation of TS Polynomials}
\label{SEC:TS}
In Section \ref{SEC:1} we posed a question about a possible description of cubic forms $[f]\in\cU$ whose Hessian locus has singularities in codimension $1$ (see Question \ref{QST:question}). First of all, let us notice that the locus in $\bP(S^3)$ we are interested in is not empty. Indeed, one can easily exhibit poynomials whose Hessian locus is reducible.

\begin{definition}
\label{DEF:TS}
Given $f\in S^d\setminus\{0\}$, we say that $f$ is a {\it Thom-Sebastiani Polynomial} (TS, for brevity) if 
\begin{equation}
\label{EQ:SVP}
f=f_1(l_0,\dots,l_k)+f_2(l_{k+1},\dots,l_n)
\end{equation}
for suitable $0\leq k\leq n-1$, $\{l_0,\dots,l_n\}$ independent linear forms and $f_1,f_2$ polynomials of degree $d$ in $k+1$ and $n-k$ variables respectively. \\
We will denote by $\cV$ the open set of smooth hypersurfaces which are not of Thom-Sebastiani type.
\end{definition}

Examples of TS polynomials are the ones whose zero locus is a cone. These are all singular, clearly. It is easy to see that, if $f$ is a TS polynomial as in Equation \eqref{EQ:SVP}, $X=V(f)$ is smooth if and only if $V(f_1(l_0,\dots, l_k),l_{k+1},\dots, l_n)$ and $V(f_2(l_{k+1},\dots, l_n),l_0,\dots, l_k)$ are smooth. This is also equivalent to ask that both $V(f_1)\subset \bP^k$ and $V(f_2)\subset \bP^{n-k-1}$ are smooth.
\smallskip

For brevity, if $\{x_0,\dots, x_n\}$ are linear forms in $\bP^n$ and if $f_1$ and $f_2$ are polynomials in $k+1$ and $n-k$ variables, respectively, let us define
$$f_1(\underline{x}):=f_1(x_0,\dots, x_k)\qquad \mbox{ and }\qquad f_2(\underline{x}'):=f_2(x_{k+1},\dots, x_{n}).$$

\begin{remark}
\label{REM:Pignolotti}
Let $f$ be a TS polynomial. Then, we can choose coordinates $\{x_0,\dots, x_n\}$ for $\bP^n$ and write $f=f_1(\underline{x})+f_2(\underline{x}')$ for suitable polynomials in $k+1$ and $n-k$ variables of degree $d$. 
Set $g_1:=f_1(\underline{x})$ and $g_2:=f_2(\underline{x}')$ so that $g_1,g_2\in S^d$ with $g_1$ that depends only on the variables $x_0,\dots, x_k$ and $g_2$ that depends only on the other variables. Then, it is clear that
$$H_f(\underline{x},\underline{x}')=\begin{bmatrix}
H_{f_1}(\underline{x}) & 0 \\
0 & H_{f_2}(\underline{x}')
\end{bmatrix}$$
so $h_f(\underline{x},\underline{x}')=h_{f_1}(\underline{x})h_{f_2}(\underline{x}')$. In particular, the Hessian variety $\cH_f$ associated to a TS polynomial is reducible and it is the union of the two cones $W_1=V(g_1)$ and $W_2=V(g_2)$. Moreover, $\Sing(\cH_f)$ has dimension $n-2$ since it contains the intersection $W_1\cap W_2$.
\end{remark}

%---------------------------------------------------------

% A polynomial $f\in S^d$, where $S=\bK[x_0,\cdots,x_{n+1}]$, is cyclic if, up to a change of coordinates it can be written as $f=x_0^d+g(x_1,\dots,x_{n})$, where $g\in\bK[x_1,\dots,x_{n+1}]_d$ (in other words, in Definition \ref{DEF:TS}, we are taking $k=1$). 

When $f=x_0^d+f_2(x_1,\dots,x_{n})$,  (in other words, in Definition \ref{DEF:TS}, we are taking $k=0$) one talks about cyclic polynomials (see also Example \ref{EX:CyclicCubics} at the end of this section). The name comes from the fact that the projection of $X=V(f)$ from the point $P_0=(1:0:\dots:0)$ to $V(x_0)\simeq \bP^{n-1}$ gives a natural structure of cyclic cover of $\bP^{n-1}$ branched along the hypersurface $V(f_2)$. For a smooth cubic $X=V(f)$, being cyclic gives a strong condition both on the associated Hessian locus and on the Jacobian ideal of $f$. Indeed, in \cite{BFGre}, it has been proved that being cyclic is equivalent to having a linear component in the Hessian variety and a point in $\cD_1(f)$. This point corresponds to an element in $J_f$ which is a square of some linear form in $S^1$, i.e. it gives a nilpotent element of order $2$ in the Jacobian ring of $f$.
\smallskip

The main purpose of this section is to give a characterization of these Thom-Sebastiani polynomials in terms of the existence of suitable linear projective spaces in some Hessian loci. In particular, we will prove Theorem $B$:

\begin{theorem}
\label{THM:charactTS}
    A polynomial $f\in \cU$ is of Thom-Sebastiani type of the form $f(x_0,\dots,x_n)=f_1(x_0,\dots,x_k)+f_2(x_{k+1},\dots,x_n)$ if and only if $\cD_{k+1}(f)$ contains a $\bP^{k}$.
\end{theorem}

First of all, if we assume that $\cD_k(f)\neq \cD_{k-1}(f)$, we can define the map
$$\varphi:\cD_k(f)\dashrightarrow \Gr(n-k,\bP^n)\qquad [x]\not\in \cD_{k-1}(f)\mapsto \iota([x])$$
whose indeterminacy locus is $\cD_{k-1}(f)$.

\begin{proposition}
\label{PROP:PhiInjective}
Assume that $\cD_k(f)\neq \cD_{k-1}(f)$. The injectivity of $\varphi$ can only fail on points along a line contained in $\cD_{k}(f)$ and cutting $\cD_{k-1}(f)$.
In particular, if $\cD_{k-1}(f)=\emptyset$ or if $\cD_{k}(f)$ does not contain lines, $\varphi$ is injective.
\end{proposition}

\begin{proof}
If $\cD_{k}(f)\setminus\cD_{k-1}(f)$ is a single point, $\varphi$ is clearly injective. Assume then that $z_1,z_2\in \cD_{k}(f)\setminus\cD_{k-1}(f)$ are distinct and that $\varphi([z_1])=\varphi([z_2])$. Then $\iota([z_1])=\iota([z_2])$ so $H_f(z_1)$ and $H_f(z_2)$ have the same kernel. Up to a change of coordinates, we can assume that $\ker(H_f(z_i))=\langle e_0,\dots, e_{n-k}\rangle$, where $\{e_0,\dots,e_n\}$ is the basis corresponding to the basis $\{y_0,\dots,y_n\}$ under the identification $\bP^n\simeq \bP(A^1)$. Hence, there exist two square matrices $A_1$ and $A_2$ of order $k$ with coefficients in $\bK$ and maximal rank such that
$$H_{f}(z_i)=\begin{bmatrix}
0 & 0\\
0 & A_i
\end{bmatrix}.
$$
Being $x\mapsto H_f(x)$ linear, $\bP(\langle z_1,z_2\rangle)\simeq \bP^{1}$ is clearly contained in $\cD_k(f)$.
\medskip

Set $p(\lambda,\mu)$ to be the polynomial $\det(\lambda A_1+\mu A_2)$. Since $\det(A_i)\neq 0$ by assumption, we have that $p$ is homogeneous of degree $k$ and non-trivial. Hence, there exists $[\lambda_0:\mu_0]$ such that $p(\lambda_0,\mu_0)=0$, i.e. $H_f(\lambda_0 z_1+\mu_0 z_2)$ has rank at most $k-1$. Thus $\bP(\langle z_1,z_2\rangle)$ cuts $\cD_{k-1}(f)$, as claimed.
\end{proof}

\begin{proposition}
\label{PROP:linearfactorsinDegLociDelux}
Let $f\in \cU$ and assume that the $(n-k-1)$-plane $\Pi=\bP(V)$ is contained in $\cD_{n-k}(f)$. Then there exists $\bP(U)\simeq \bP^k$ in $\cD_{1+k}(f)$. Moreover, for all $[u]\in \bP(U)$, one has that $\Pi\subseteq \iota([u])$ with equality holding for $[u]$ general. 
\end{proposition}

\begin{proof}
By assumption, one has $\Rank(H_f(v))\leq n-k$ for all $v\in V\setminus\{0\}$. 
One can see (\cite{BFP2}) that the quadric of $\bP^n$ given by the vanishing of the polynomial $\pa{v}(f)$ is represented by the square symmetric matrix $H_f(v)$. Then the singular locus of the quadric $V(\pa{v}(f))$ contains a $\bP^{k}$. By setting $W=\{\pa{v}(f)\}_{v\in V}$, we can then observe that $|W|$ is a linear subsystem of dimension $n-k-1$, since the map $v\mapsto \pa{v}(f)$ is injective as $V(f)$ is smooth (it would have been enough to ask that $V(f)$ is not a cone).
\medskip

Let $J=J_f$ be the Jacobian ideal of $f$. Since $|W|\subset |J^2|$ and $J^2$ is spanned by a regular sequence, we have that $B:=\BL(|W|)$ has pure dimension $k$. Indeed, being $B$ cut by $n-k$ quadrics, we have that $\dim(B)\geq k$. On the other hand, if there were a component of $B$ with dimension at least $k+1$, then, we would be able to complete a basis of $W$ in such a way that $J^2$ is not spanned by a regular sequence.
\medskip

By Bertini's theorem the general element of $|W|$ is smooth away from $B$, which has pure dimension $k$. On the other hand, as observed before, all quadrics of $|W|$ have a $\bP^k$ contained in the singular locus. Therefore, there exists a component of $B$ which is a $\bP^k$. Since $B$ has dimension $k$, it contains at most a finite number of $\bP^k$: there exists a component of $B$ which is a $\bP^k$ contained in the singular locus of all the elements of $|W|$. Let $\bP(U)\simeq \bP^k$ be this linear space.
\medskip

Consider $[u]\in \bP(U)$. Since all the quadrics $V(\pa{v}(f))$ parametrized by $W$ are singular along $\bP(U)$ we have 
$$\pa{u}\pa{v}(f)=0\qquad \mbox{ for all }[v] \in \bP(V).$$
This implies that $\pa{v}(\pa{u}(f))=0$ for all $[v]\in \bP(V)$: $V(\pa{u}(f))$ is a quadric whose singular locus contains the $(n-k-1)$-plane $\Pi=\bP(V)$. Therefore, $H_f(u)$ has rank at most $k+1$ and thus $\bP(U)\subseteq \cD_{k+1}(f)$.\\

Finally, notice that $\dim(\cD_{k}(f))\leq k-1$ by Proposition \ref{PROP:Comp_Of_Gamma} so $\bP(U)\simeq \bP^k$ cannot be contained in $\cD_{k}(f)$. Hence, for the general point $[u]\in \bP(U)$ the singular locus of $V(\pa{u}(f))$ is exactly the $(n-k-1)$-plane $\Pi$. In other terms, we have $\iota([u])=\bP(V)$ for $[u]\in \bP(U)$ general. 
\end{proof}

\begin{corollary}
\label{COR:Linear_planes_imply_Dk_not-empty}
Let $f\in \cU$ and assume that there exists $k\geq1$ such that $\cD_{n-k}(f)$ contains a $(n-k-1)$-plane. Then $\cD_{k}(f)\neq \emptyset$.
\end{corollary}

\begin{proof}
Assume that $\bP(V)$ is a $(n-k-1)$-plane in $\cD_{n-k}(f)$. By Proposition \ref{PROP:linearfactorsinDegLociDelux} we have that there exist $\bP^k\simeq\bP(U)\subseteq\cD_{k+1}(f)$ such that $\bP(V)\subseteq \iota([u])$ for $[u]\in \bP(U)$ with equality holding for $[u]$ general. Since, in this case, $\cD_{k+1}(f)$ and $\cD_k(f)$ don't coincide, for dimensional reason (by Proposition \ref{PROP:Comp_Of_Gamma}), we can define the map $\varphi:\cD_{k+1}(f)\setminus \cD_{k}(f)\to G(n-k-1,\bP^n)$. Then, the injectivity of $\varphi$ fails on two general points of $\bP(U)$. Finally, by Proposition \ref{PROP:PhiInjective}, we have that $\bP(U)\cap \cD_{k}(f)\neq \emptyset$ as claimed. 
\end{proof}

We can now prove Theorem \ref{THM:charactTS}:

\begin{proof}
    First of all, let us assume that for a fixed $k\geq 0$ there exists $\bP^k\simeq\bP(V)\subseteq\cD_{k+1}(f)$. Then, by Proposition \ref{PROP:linearfactorsinDegLociDelux} there also exists $\bP(U)\simeq \bP^{n-k-1}$ contained in the locus $\cD_{n-k}(f)$. Moreover, we know that for all $[u]\in\bP(U)$ the projective space $\bP(V)$ is contained in $\iota([u])$. This means that for every $[u]\in \bP(U)$ and $[v]\in \bP(V)$ we have $uv=0$, with the identification $\bP^n=\bP(A^1)$. Let us notice that the spaces $\bP(U)$ and $\bP(V)$ are skew in $\bP^n$ and of complementary dimension. Indeed, if their intersection was non-trivial, we could find a point $[x]\in\bP(V)\cap\bP(U)$: from above, we would obtain $x^2=0$, against the smoothness of $V(f)$. We can then consider for $\bP^n$ a coordinate system $x_0,\dots,x_n$ where $\bP(V)=V(x_0,\dots,x_{k})$ and $\bP(U)=V(x_{k+1},\dots,x_n)$. Up to a change of coordinates, we can then write the polynomial $f$ with respect to these variables: since, by construction, $x_ix_j=0$ in $A^2$ for every $i=0,\dots,k$ and $j=k+1,\dots,n$, we get the claim.
    \medskip

    Let us now assume that $f$ is TS: as in Remark \ref{REM:Pignolotti} we can write it (up to a possible change of coordinates) as $f=f_1(\underline{x})+f_2(\underline{x}')$. As already observed, the Hessian matrix of $f$ is of the form
    $$H_f(\underline{x},\underline{x'})=\begin{bmatrix}
    H_{f_1}(\underline{x}) & 0 \\
    0 & H_{f_2}(\underline{x'})
    \end{bmatrix}.$$
    Hence, by defining $\bP(V):=V(x_0,\dots,x_k)\simeq \bP^k$ one easily sees that $\Rank(H_f(v))\leq k+1$ for every $[v]\in\bP(V)$, i.e. $\bP(V)\subseteq\cD_{k+1}(f)$ as claimed.
\end{proof}

To end this section, let us present some key examples of TS polynomials. 

\begin{example}[Cyclic cubics]
\label{EX:CyclicCubics}
The simplest examples of TS polynomials are the cyclic polynomials. We recall that a polynomial $f\in S^d$, where $S=\bK[x_0,\cdots,x_{n}]$, is cyclic if, up to a change of coordinates, it can be written as $f=x_0^d+g(x_1,\dots,x_{n})$, where $g\in\bK[x_1,\dots,x_{n}]_d$.
\smallskip

As observed before, $X=V(f)$ is smooth exactly when $V(g)\subset \bP^{n-1}$ is smooth and $h_f=d(d-1)x_0^{d-2}\cdot h_g(x_1,\dots,x_n)$ so the Hessian variety splits as the union of a hyperplane and a hypersurface of degree $n(d-2)$, namely
$$H=V(x_0)\qquad\mbox{ and }\qquad  W=V(h_g(x_1,\dots, x_n)).$$
Notice that $W$ doesn't need to be irreducible, but this is the case if $g$ is general (and $n\geq 3$). Under the identification $\bP^{n-1}\simeq V(x_0)$, we can say that the Hessian loci $\cD_k(g)$ live in $H=V(x_0)$. We denote by $\hat{\cD}_k(g)$ the cone over $\cD_k(g)\subseteq V(x_0)$ with vertex the point $P_0$. For example, one has $W=\hat{\cD}_{n-1}(g)$. Then, one can prove that
\begin{equation}
\label{EQ:HessianLoci_Cyclic}
\cD_k(f)=\cD_k(g)\cup \hat{\cD}_{k-1}(g) \cup \{P_0\}.
\end{equation}

It is well known that the general cubic surface $S=V(g)\subseteq \bP^3$ has an irreducible Hessian variety which is a quartic with $10$ nodes as the only singularities. This was known already by B. Segre (see \cite{Seg}) but one can also refer to the more recent \cite{DVG}. In particular
$$\cD_3(g)=\cH_g\qquad \cD_2(g)=\Sing(\cH_g)=\{Q_1,\dots, Q_{10}\}\qquad \cD_1(g)=\emptyset.$$
Using this observation and Equation \eqref{EQ:HessianLoci_Cyclic}, one can describe the stratification given by the Hessian loci of a general cyclic cubic threefold $X=V(f)$:
$$\cD_5(f)=\bP^4\qquad \cD_4(f)=\cH_f=H\cup W\qquad \cD_3(f)=\cH_g \cup \bigcup_{i=1}^{10}\langle P_0,Q_i\rangle$$ $$\cD_2(f)=\{P_0\}\cup \{Q_1,\dots, Q_{10}\}\qquad \cD_1(f)=\{P_0\}.$$
Among these, only $H,W\subset \cD_4(f)$, $H\cap W=\cH_g\subset \cD_{3}(f)$ and $\{P_0\}\subset \cD_1(f)$ give irreducible components of $\Gamma_f$ (this will be clear from Lemma \ref{LEM:triangfam}).
\end{example}

Since smooth binary cubic forms can be written as sum of $2$ cubes, every TS polynomial in $n+1$ variables, with $n\leq4$, is necessarily cyclic (see \cite{BFGre} for details). Let us now describe a new phenomenon arising in $\bP^5$.

\begin{example}[A TS cubic which is not cyclic]
\label{EX:cub+cub}
Let $g_1,g_2\in S_w=\bK[w_0,w_1,w_2]$ be such that $V(g_1)$ and $V(g_2)$ are smooth cubic curves in $\bP^2$ which are not the Fermat curve. This is equivalent to ask that $V(g_i)$ is a cubic whose Hessian $V(h_{g_i})$ is irreducible.
\smallskip

A smooth cubic fourfold $X$ of TS type is not cyclic if and only if, up to a change of coordinates, it is defined by a polynomial $f=g_1(x_0,x_1,x_2)+g_2(x_3,x_4,x_5)$. From the point of view of moduli, such fourfolds form 
%an open dense subset of 
a dimension $2$ variety in the moduli space of smooth cubic fourfolds. 
%\fbox{\Large \color{red} Check }
\smallskip

Consider the following varieties:
$$W_1=V(h_{g_1}(x_0,x_1,x_2))\quad W_2=V(h_{g_2}(x_3,x_4,x_5))\quad \Pi_1=V(x_0,x_1,x_2) \quad \Pi_2=V(x_3,x_4,x_5)$$
$$C_1=\Pi_2\cap W_1\simeq V(h_{g_1})\qquad C_2=\Pi_1\cap W_2\simeq V(h_{g_2})\qquad J=J(C_1,C_2),$$
where $J(C_1,C_2)$ is the join variety of $C_1$ and $C_2$, namely the union of all lines joining a point of $C_1$ and a point of $C_2$. If $\{i,j\}=\{1,2\}$, the variety $W_i$ is a cone over $C_j$ with vertex $\Pi_i$ and is irreducible by our assumption on the curves $V(g_1)$ and $V(g_2)$. Being $f$ a TS polynomial, one has that $\cH_f$ is indeed reducible. More precisely, since $h_f=h_{g_1}(x_0,x_1,x_2)h_{g_2}(x_3,x_4,x_5)$, one has that $$\cD_5(f)=\cH_f=W_1\cup W_2.$$ 
The other strata of the stratification induced by $f$ are
$$\cD_{4}(f)=\Sing(\cH_f)=J\qquad \cD_3(f)=\Pi_1\cup \Pi_2\qquad \cD_2(f)=C_1\cup C_2$$
whereas $\cD_1(f)=\emptyset$ as $X$ is not cyclic (by the results in \cite{BFGre}).
\smallskip

It is worth to highlight two facts. First of all, $\Pi_1$ and $\Pi_2$ are two $2$-planes contained in $\cD_3(f)$. These are exactly the $k$-planes contained in $\cD_{k+1}(f)$ whose existence is guaranteed by Theorem \ref{THM:charactTS} since $f$ is a TS polynomial. Moreover, note that for all $k\in \{2,3,4,5\}$, the dimension of $\cD_k(f)$ equals $k-1$, i.e. the maximum predicted by Proposition \ref{PROP:Comp_Of_Gamma}. In particular, $\Gamma_f$ splits as the union of $7$ irreducible fourfolds (this follows from Lemma \ref{LEM:triangfam}).
\end{example}

% ---------------------------------------------------------
% ---------------------------------------------------------
% ---------------------------------------------------------
% ---------------------------------------------------------
% ---------------------------------------------------------

\section{Families of triangles of high dimension}
\label{SEC:FamiliesOfTriangles}

In this section, we focus on the study of suitable families of triangles for $\cH_f$ arising naturally, as we will see in a moment, when $\Sing(\cH_f)$ exceeds the expected dimension. Moreover, we prove Theorem $C$.
\smallskip

Let us now set some notations and prove some technical results. 

We recall that, given a smooth cubic $V(f)\subseteq \bP^n\simeq \bP(A^1)$, a family of triangles for $\cH_f$ is a subvariety $\cF$ of
$$\cT=\{([x],[y],[z])\in (\bP(A^1))^3\,|\, xy=yz=xz=0\}.$$
We will denote by $\pi_i$ the natural projections from $\cF$ on the factors. For simplicity, if $\cF$ is a family of triangles for $\cH_f$, we will set $Y_i=\pi_i(\cF)$ for $i\in \{1,2,3\}$. Notice that $\dim(Y_i)\leq n-2$ since $Y_i\subseteq \cD_{n-1}(f)$ which has dimension at most $n-2$ by Proposition \ref{PROP:Comp_Of_Gamma}.
\smallskip

Moreover, recall that if $[f]\in \cU$, by Proposition \ref{PROP:Comp_Of_Gamma}, all components of $\Gamma_f$ come from the Hessian loci of $f$. More precisely, if $Z$ is an irreducible component of $\cD_{k}(f)$ of dimension $k-1$, there exists an irreducible component $\Gamma_f$ which dominates $Z$ by first projection. We will denote by $\tilde{Z}$ such component.

\begin{lemma}\label{LEM:triangfam}
    Let $[f]\in\cU$ and assume $k\in \{1,\dots, n-1\}$. Consider an irreducible component $W$ of $\cH_f$ and an irreducible component $Z$ of $\cD_{k}(f)$ of dimension $k-1$ which is contained in $W$. Then $\tilde{Z}\cap \tilde{W}$ dominates $Z$ via the first projection $\pr_1$.
    In particular, there exists a family of triangles $\cF$ for $\cH_f$ of dimension at least $k-1$.    Moreover, if $k=n-1$, then every family as above has dimension exactly $n-2$.
\end{lemma}

\begin{proof}
Notice that $Z$ is not contained in $\cD_{k-1}(f)$ by Proposition \ref{PROP:Comp_Of_Gamma} as we are assuming $\dim(Z)=k-1$. Hence, the general point $z\in Z$ lies in 
$\cD_k(f)\setminus \cD_{k-1}(f)$ and thus $\iota([z])\simeq \bP^{n-k}$. Since the general fiber of 
$\pr_1|_{\tilde{Z}}:\tilde{Z}\to Z$ has dimension $n-k$ by construction, one has that the whole fiber 
$\pr_1^{-1}([z])=\{[z]\}\times \iota([z])$
is contained in $\tilde{Z}$. On the other hand, 
$\pr_1|_{\tilde{W}}:\tilde{W}\to W$ is surjective and $Z\subseteq W$ by assumption so there exists at least a point $p=([z],[y])$ of the whole fiber $\pi_1^{-1}([z])$ with $p\in \tilde{W}$. Then $p \in U=\tilde{W}\cap \tilde{Z}$ and $\pr_1|_{U}:U\to Z$ is such that $\pr_1|_{U}(p)=z$. In particular, $\pr_1|_{U}$ dominates $Z$.
\smallskip

By the above argument, we have that $\tilde{W}$ and $\tilde{Z}$ are irreducible components of $\Gamma_f$ which meet in a variety $U$ of dimension at least $k-1$. Then, we have a family of dimension $k-1$ since this variety is contained in $\Sing(\Gamma_f)$ by construction and each point yields (at least) a triangle by Lemma \ref{LEM:Triangles_and_Singularities}. 
\smallskip

Let $Z\subseteq\cD_{n-1}(f)$ be an irreducible component of dimension $n-2$ and let $\cF$ be a family of triangles dominating $Z$ via $\pi_1$, so that $\dim(\cF)\geq n-2$. By construction, the general point $[x]$ of $Z$ is such that $\iota([x])\simeq \bP^1$ so $\pi_1^{-1}([x])\subset\{[x]\}\times \bP^1\times \bP^1$. If the general fiber $\pi_1^{-1}([x])$ has positive dimension we would have that $\pi_1^{-1}([x])\cap \{[x]\}\times \Delta_{\bP^1}$ is not empty, thus giving rise to a singular point of $V(f)$. Then the general fiber of $\pi_1$ has dimension $0$ and thus $\dim(\cF)=n-2$.
\end{proof}

\begin{remark}
\label{REM:Non_Normal_BigFamilyofTriangles}
If $\cH_f$ is not normal, we have that there exists at least a family of triangles of dimension $n-2$. Indeed, we have that the singular locus of $\cH_f$ has dimension $n-2$ and equals $\cD_{n-1}(f)$ by Theorem \ref{THM:fundamental}. Hence, given an irreducible component $Z$ of $\Sing(\cH_f)$ of dimension $n-2$, we have that $Z$ yields a family of triangles of dimension $n-2$ as a consequence of Lemma \ref{LEM:triangfam}.
\end{remark}

\begin{lemma}
\label{LEM:bound}
    Let $X=V(f)\subset \bP^n$ be a smooth cubic hypersurface and let
    $\cF$ be an irreducible family of triangles. Then the following hold:
    \begin{enumerate}[(a)]
        \item If $\dim(Y_i)\geq 1$ and $Y_i\subset X$, then $\dim(Y_i)\leq n-3$;
        \item If $n\geq3$ and $\dim(Y_i)=n-2$ for some $i$, then no projection has dimension $0$ unless $V(f)$ is of Thom-Sebastiani type.
    \end{enumerate}
\end{lemma}

\begin{proof}
    For $(a)$, w.l.o.g. we can assume $\dim(Y_1)\geq1$ and $Y_1\subset V(f)$. If $T=([x],[y],[z])\in \cF$ is a general triangle, then the differential
    $d\pi_{1,T}:T_{\cF,T}\rightarrow T_{Y_1,[x]}$ is surjective and it sends a tangent vector $(x',y',z')$ to $x'$. By Lemma \ref{LEM:TGVec_Triang_First_Order}, $x'\in\Ann_{A^1}(y^2,z^2)/\langle x\rangle$. Moreover, since $Y_1\subset X$, by Proposition \ref{PROP:descrapol} we have $X=\{[x]\,|\, x^3=0\}$ so $T_{X,[x]}=\Ann_{A^1}(x^2)/\langle x\rangle$.
    Hence, 
    $$T_{Y_1,[x]}\subseteq \Ann_{A^1}(x^2,y^2,z^2)/\langle x\rangle.$$
    Since $T$ is a triangle, by Lemma \ref{LEM:squareindep} one has that $\langle x^2,y^2,z^2\rangle$ has dimension $3$ and thus, by the Gorenstein duality,  $\dim(\Ann_{A^1}(x^2,y^2,z^2))=n+1-3=n-2$. Moreover, being $[x]\in X$, one has $x\in \Ann(x^2,y^2,z^2)$ so $\Ann_{A^1}(x^2,y^2,z^2)/\langle x\rangle$ has dimension $n-3$.
    \smallskip
    
    For $(b)$, w.l.o.g. assume that $Y_3$ is of dimension $n-2$. By contradiction, let us assume that $Y_1=\{[x]\}$ and that $f$ is not of TS type. Hence $Y_3\subseteq\iota([x])$ and this implies that $\iota(x)=\bP^s$ with $s\in\{n-2,n-1\}$. 
    
    Then, we would get $\iota([x])=\bP^{n-2}=Y_3\subseteq\cD_{n-1}(f)$ and $[x]\in\cD_1(f)$, respectively. Both cases yield a contradiction by Theorem \ref{THM:charactTS}.
\end{proof}

Before proving Theorem $C$ (Theorem \ref{THM:moon}) let us focus on (families of) triangles with all vertices on the cubic $X$. These are linked to families of $2$-planes in the cubic hypersurface:

\begin{remark}
\label{RMK:planesin34folds}
First of all, recall that if $T=([x],[y],[z])$ is a triangle for $\cH_f$, then $\left\langle [x],[y],[z]\right\rangle=\bP^2$, by Lemma \ref{LEM:squareindep}. If we assume moreover, that all the vertices of $T$ belong to the cubic hypersurface $X$ we have $x^3=y^3=z^3=0$ and also $xy=yz=zx=0$; this implies that the $2$-plane is actually contained in $X$. 
Hence, a triangle with three vertices on $X$ can not exist, if $X$ is a smooth cubic hypersurface of dimension at most $3$.
Furthermore, since on smooth cubic fourfolds one has at most a finite number of $2$-planes  (see, for example, \cite{DIO}), by Proposition \ref{PROP:injtriangles} we can have at most a finite number of triangles with all the vertices on the cubic $X$.
\end{remark}

\begin{lemma}
\label{LEM:contr_fibers}
Assume that $\cF$ is a family of triangles for $\cH_f$ with $\dim(Y_1)=\dim(\cF)>\dim(Y_3)$. Then, none of the fibers of the projection $\pi_3$ can be contracted to points via $\pi_1$.  
\end{lemma}

\begin{proof}
This follows from a more general fact: if $g:X\to Z$ is a surjective morphism between irreducible varieties and $f:X\to Y$ is a morphism, the locus
$$A=\{ z\in Z \,|\, \dim(f(g^{-1}(z)))=0\}$$
is an open Zariski set of $Z$. In order to prove this let us consider the map $F=(f,g):X\rightarrow Y\times Z$ and denote by $X'\subseteq Y\times Z$ the image of $X$ under the morphism $F$. Moreover, let $p_1$ and $p_2$ be the two projections from $X'$ to $Y$ and $Z$ respectively. It is then clear that 
$$A=\{z\in Z \ | \ \dim(p_1(p_2^{-1}(z)))=0\}=\{z\in Z \ | \ \dim(p_2^{-1}(z))=0\}$$
and thus it is an open subset of $Z$. 
In our situation, if we assume that a fiber of $\pi_3$ is contracted to points via $\pi_1$ then the same is true for the general one, contradicting the assumption $\dim(Y_1)=\dim(\cF)$.
\end{proof}

We are now ready to prove Theorem $C$, one of the main ingredient in the proof of Theorem $A$.

\begin{theorem}
    \label{THM:moon}
    Assume $n\leq 5$ and consider $[f]\in\cV$. If $\cF$ is an irreducible family of triangles for $\cH_f$ with $\dim(\cF)=\dim(\pi_1(\cF))=n-2$, then the general element in $\cF$ is such that none of its vertices belongs to $X=V(f)$.
\end{theorem}

\begin{proof}
Notice that the proof follows easily if none of the three projections of $\cF$ is contained in $X$. Indeed, in this case, $\pi_i^{-1}(Y_i\cap X)$, the locus where the triangles of $\cF$ have the $i$-th vertex on the cubic $X$, is a proper closed subset of $\cF$. We would like to show that also for $n\leq 5$ this is indeed the only possible case, i.e. no projection of $\cF$ can be contained in $X$. For $n=3$, this is an easy consequence of Lemma \ref{LEM:bound} so we can assume $n\geq 4$. 
\smallskip

W.l.o.g, we can assume that $\pi_1$ is generically finite (and thus, by Lemma \ref{LEM:bound}, $Y_1$ is not contained in $X$) and, by contradiction, that $Y_3\subset X$. Notice that, under these assumptions, $Y_2$ is not contained in $X$. Otherwise, $\pi_1^{-1}(Y_1\cap X)$ would be an $(n-3)$-dimensional family of triangles with all the vertices contained in $X$. Then, we would have a contradiction as observed in Remark \ref{RMK:planesin34folds}. For brevity, set $\cF_c=\cF\cap (X\times X\times X)$, i.e. $\cF_c$ is the locus of the triangles of $\cF$ with all $3$ vertices on the cubic hypersurface $X$.
\smallskip

{\bf Assume that $n=4$}. By Lemma \ref{LEM:bound}, since $Y_3\subset X$ and $Y_1$ has dimension $2=n-2=\dim(\cF)$, we have that $\dim(Y_3)=1$. Then all the fibers of $\pi_3$ have pure dimension $1$. The dimension of $Y_2$ is either $1$ or $2$. If the dimension of $Y_2$ is $1$, all the fibers of $\pi_2$ are curves too and $Y_2\cap X$ is not empty. Consider $[y_0]\in Y_2\cap X$ and its fiber $C=\pi_2^{-1}([y_0])$. By Lemma \ref{LEM:contr_fibers}, $C$ can not be contracted by $\pi_1$ so $\pi_1(C)$ is a curve. Then, $\pi_1(C)\cap X$ is not empty and we can consider a point $[x_0]$ in this intersection. Hence, any element in $\pi_1^{-1}([x_0])\cap C\neq \emptyset$ is a triangle in $\cF_c$. This is impossible by Remark \ref{RMK:planesin34folds}. Then we have necessarily $\dim(Y_2)=2$.
\smallskip

Let $Y$ be an irreducible component of $Y_2\cap X$. Being $Y_2\not\subseteq X$  and of dimension $2$, $Y$ is a curve and there exists an irreducible component $C$ in $\pi_2^{-1}(Y)$ of dimension $1$ dominating $Y$ via the second projection. If either $\pi_1(C)$ is a curve or $\pi_1(C)=[x_0]$ with $[x_0]\in X$ we have an element in $\cF_c$, so the only possible case is $\pi_1(C)=[x_0]$ with $[x_0]\not\in X$.
\smallskip

Looking at the third projection we have that $\pi_3(C)$ either is a point $[z]$ or it coincides with $Y_3$. Moreover, let us observe that $\iota([x_0])\simeq \bP^s$ with $s\in\{1,2\}$ such that $Y$ and $\pi_3(C)$ are contained in $\iota([x_0])$. To rule out the first case, namely $\pi_3(C)=[z]$, first of all, observe that $[z]\not\in Y$, otherwise we would have a singular point for the cubic $X$. Then $s$ is forced to be $2$. Moreover, by construction, we have $yz=0$ and $y^3=z^3=0$ for any $[y]\in Y$: by reasoning as in Remark \ref{RMK:planesin34folds} all the lines $\langle [y],[z]\rangle$ lie in $X$. This implies that the whole $\iota([x_0])$ is contained in $X$, but this is not possible.
\smallskip

For the remaining case, we have that $\pi_3(C)=Y_3$ and, by construction, both the curves $Y$ and $Y_3$ are contained in $\iota([x_0])$. Then, if $s=1$ we necessarily have $Y=Y_3\simeq \bP^1$ and $C$ is a family of triangles of dimension $1$ in $\{[x_0]\}\times Y\times Y$. This yields a contradiction since $C$ has to meet $\{[x_0]\}\times \Delta_{Y}$, thus giving a singular point for $X$. Hence we have necessarily $s=2$ and we can assume $Y\neq Y_3$ or $Y=Y_3\neq \bP^1$. In both cases, as done above, considering the lines $\langle [y],[z]\rangle$ with $[y]\in Y, \ [z]\in Y_3$ and $yz=0$, we get that the $2$-plane $\iota([x_0])$ is contained in the cubic threefold $X$, which is not possible. 
\smallskip

{\bf Assume now that $n=5$}. We are working in the following framework: $\cF$ is an irreducible $3$-dimensional family of triangles with $\pi_1$ generically finite, $Y_1,Y_2$ not contained in $X$ and $Y_3\subseteq X$ (this will lead to a contradiction). Being $Y_3\subseteq X$ by assumption, $\cF_c$ is cut out from $\cF$ by two divisors so its expected dimension is $n-4=1$. Then, either $\cF_c$ is empty, or $\dim(\cF_c)\geq 1$. On the other hand, as observed in Remark \ref{RMK:planesin34folds}, under the above-mentioned hypotheses we have that $\dim(\cF_c)\leq 0$ so $\cF_c$ is necessarily empty. We will now prove that $\cF_c$ is not empty, thus leading to a contradiction.
\smallskip

First of all, notice that $\dim(Y_2)\in \{1,2,3\}$ by Lemma \ref{LEM:bound}. If $\dim(Y_2)\leq 2$ the general fiber of $\pi_2$ has positive dimension and cannot be contracted to a point by $\pi_1$ by Lemma \ref{LEM:contr_fibers}. Then, its image meets $X$ and thus we produce a triangle in $\cF_c$ as the analogous case for the threefold. We can then assume $\dim(Y_2)=3$ so that $\pi_2$ is generically finite as $\pi_1$.
\smallskip

Denote by $Y$ an irreducible component of $Y_2\cap X$. The preimage $\pi_2^{-1}(Y)$ has dimension $2$ and we can consider an irreducible component $S\subset\pi_2^{-1}(Y)$ dominating $Y$. If $\pi_1(S)$ is not a point or is a point on the cubic fourfold, as in the threefold case one can easily construct an element in $\cF_c$: we can assume $\pi_1(S)=[x_0]\not \in X$. As done in the previous case, we have that $\iota([x_0])\simeq \bP^s$ containing $Y$ and $\pi_3(S)$ (thus $s\in\{2,3\}$ since $X$ is not of TS type), where $\pi_3(S)\subseteq Y_3$ can be either a point $[z]$, a curve $C\not\subset Y_3$ (with $\dim(Y_3)=2$) or the whole $Y_3$ (with $\dim(Y_3)\in\{1,2\})$. To conclude the proof, let us study these distinguished cases. 
\smallskip
\begin{itemize}
    \item {\bf $\pi_3(S)=[z_0]$}: since $[z_0]\not\in Y$ (otherwise we would have a singular point in $X$), $s$ is forced to be equal to $3$. Considering again the lines $\langle [y],[z_0]\rangle$ with $[y]$ varying in $Y$ we have that the whole $3$-space $\iota([x_0])$ is contained in the smooth cubic fourfold $X$, which is clearly not possible.  
    \item {\bf $\pi_3(S)=C$}: In this case one has $\iota([x_0])\simeq \bP^3$, since, otherwise, we would have $C\subseteq Y=\bP^2$ and $S\subseteq \{[x_0]\}\times Y\times Y$ and then a singular point for $X$ as in a previous case. If we assume $C\not\subset Y\subset \iota([x_0])=\bP^3$, then one easily sees that the lines $\langle [y],[z]\rangle$ with $[y]\in Y, [z]\in C$ and $yz=0$ cover $\iota([x_0])$: we have a contradiction since we would have a $3$-projective space in $X$. The only remaining case to analyse is then the one where $C\subset Y\subset \bP^3$ with $Y$ surface which is not a $\bP^2$. In this case, $S$ is a surface in $\{[x_0]\}\times Y\times C$ with the projections $p_2=\pi_2|_S$ and $p_3=\pi_3|_S$ which are surjective. Then, for all $[z]\in C$, $p_3^{-1}([z])$ has pure dimension $1$. Let $[z]$ be a point in $C$ and let $D$ be an irreducible component of one of those fibers. For all $[y]\in p_2(D)$ we have $yz=y^3=z^3=0$ and $[z]\not\in p_2(D)$ so the joint variety $J(p_2(D),[z])$ has dimension $2$, is a cone with vertex $[z]$ and is completely contained in $\iota([x_0])\cap X$. Since, $\iota([x_0])\simeq \bP^3$ cannot be contained in $X$, these cones have to vary at most discretely, when $[z]$ moves in $C$. Notice that $Y$ lies, by construction, in the union of these cones so $[z]$ is in the vertex $\Ve(Y)$ of $Y$. Then $C\subseteq \Ve(Y)$ and this forces $Y$ to be a $\bP^2$, which is against our assumptions.
    
    \item {\bf $\pi_3(S)=Y_3$}: As in the previous case, we necessarily have $\iota([x_0])\simeq \bP^3$ containing both the surface $Y$ and $\pi_3(S)$, which can be either a curve or a surface. We can assume, moreover $Y_3=\pi_3(S)\subseteq Y$, since otherwise, proceeding as above, we would have $\iota([x_0])\subset X$. In particular, $Y$ is not a $2$-plane. If $Y_3$ is a curve we can obtain a contradiction as in the previous case by considering the cones with vertex $[z]\in Y_3$ spanned by the curves in $Y$ whose elements annihilate $[z]$. Hence, we can assume $Y=Y_3$ (and thus $\pi_3$ is generically finite). By construction, for any element $[y]\in Y$, there exists at least one element $[z]\in Y$ such that $yz=0$, hence, again, the line $\left\langle[y],[z]\right\rangle$ is contained in $\iota([x_0])\cap X$. 
    If, for $[y]$ general, at least one of these lines is not contained in $Y$, one can see that the whole $3$-space $\iota([x_0])$ is contained in $X$, yielding a contradiction. We can then assume that for $[y]\in Y$ general, the above-mentioned lines are contained in $Y$. Our aim is now to show that $Y\simeq\bP^2$, against our assumptions. First of all, let us show that the union $\cL$ of these lines as subset in the Grassmannian $\Gr(1,\bP^3)$ has dimension $2$. 
    If, by contradiction, $\dim(\cL)=1$, this would mean that for all $\ell\in\cL$ and for all $[y]\in \ell$ 
    there exists $[z]\in \ell$ with $yz=0$. In other words, this would yield a correspondence in $\bP^1\times\bP^1$, which intersects the diagonal $\Delta_{\bP^1}$ non trivially: the cubic $X$ would be singular, which is not possible. 
    Let us finally consider the incidence variety 
    $$\Psi:=\{(y,\ell) \ | \ y\in\ell\in\cL\}\subset Y\times \Gr(1,\bP^3)$$
    and denote by $\psi_1$ and $\psi_2$ the two projections. We have just shown that $\Ima(\psi_2)=\cL$ and, moreover, it is clear that if $\ell\in\cL$, then $\psi_2^{-1}(\ell)$ is described by $\ell$ itself, hence such a fiber has dimension $1$. Then $\dim(\Psi)=3$ and looking at the first projection $\psi_1$, we have that there exist infinitely many lines in $\cL$ contained in $Y$ and passing through the general point $[y]\in Y$. Hence, $Y$ has to be a cone with $[y]\in\Ve(Y)$: from the generality of $[y]$, it follows that $Y\simeq\bP^2$ as claimed.    
\end{itemize}
\end{proof}

\begin{remark}
Observe that if $n=2$ the hypotheses of Theorem can't be realized since the existence of a triangle implies that the locus $\cD_1(f)$ is non-empty. Hence, by Theorem \ref{THM:charactTS} the cubic $f$ is of TS type, against our assumption. 
\end{remark}

% --------------------------------------------------------
% --------------------------------------------------------

\section{Proof of main theorem: the cubic threefold case}
\label{SEC:star_CubicThreefoldCase}

In this section we state and begin to prove the main result of this article, namely

\begin{theorem}[Theorem A]
\label{THM:star}
Assume that $2\leq n\leq 5$ and consider $f\in \bK[x_0,\dots,x_n]$ defining a smooth cubic. Then,
the singular locus of the Hessian hypersurface $\cH_f\subset\bP^n$ has the expected dimension if and only if $f$ is not of TS type. In particular, $f\in \cV$ if and only if $\cH_f$ is irreducible and normal.    
\end{theorem}

As we observe now, the assumptions on the degree and the smoothness of the hypersurface $X=V(f)$ are essential.
\begin{remark}
\label{REM:counterexamples}
    Let us stress that the result stated in Theorem \ref{THM:star} is false for smooth hypersurfaces of degree $d\geq 4$ and for non-smooth cubics. We provide here two simple examples proving these claims.
    \begin{itemize}
        \item Let $f(x,y,z)=x^4+y^4+z^4+x(y^3+z^3)$ and consider $C=V(f)$. Then one easily sees that $C$ is a smooth quartic plane curve and that
        $$h_f=\lambda\cdot yz\left( 8x^4 + 16x^3(y + z) + 32x^2yz - x(y^3+z^3) - 2yz(y^2 + z^2)\right).$$
        The quartic factor in the above factorization of $h_g$ yields a smooth quartic by the Jacobian criterion and thus, an irreducible one. This also implies that $g$ is not of TS type since, otherwise, we would have a completely decomposable hessian polynomial: there are smooth hypersurfaces of degree $d\geq 4$, which are not of TS type, with reducible Hessian variety.
        \item Let $f(x_0,x_1,x_2,x_3)=x_0x_1^2+x_1x_2^2+x_2x_3^2$ and consider the cubic surface $S=V(f)$. One can see that $S$ is an irreducible cubic surface whose singular locus coincides with the point $p_0=(1:0:0:0)$, which is a singularity of type $D_5$. Its associated Hessian variety is a reducible and non-reduced quartic surface $V(x_1^2(x_1x_2-x_3^2))$. Notice that the quadratic factor of the hessian polynomial $h_f$ is irreducible so, reasoning as in the previous example, one can see that $f\neq f_1(z_0,z_1)+f_2(z_2,z_3)$ for suitable coordinates $\{z_0,\dots,z_3\}$. With a direct and easy computation, also the cyclic case is ruled out: $f$ is not of TS type although its hessian is reducible (and thus non-normal) and non-reduced. The same phenomenon happens for the cuspidal cubic curve (see, for example, \cite{CO}). 
    \end{itemize}
    Nevertheless, not every type of singularity gives the same behaviour as in the last example above. Indeed, the nodal cubic curve and the $1$-nodal cubic surface $V(x_0(x_1^2+x_2^2+x_3^2)+x_1^2x_3+x_2x_3^2)$ have irreducible and normal associated Hessian variety (and thus they are not of TS type). One can easily construct examples of $1$-nodal cubic threefolds and fourfolds with the same property.
\end{remark}

Going back to the case of smooth cubic hypersurfaces, we make the following:
\begin{conjecture}
    The same result stated in Theorem \ref{THM:star} holds for smooth cubic hypersurfaces in $\bP^n$ for every $n\geq 2$. 
\end{conjecture}

The techniques used in the proof of Theorem \ref{THM:star} do not seem to adapt to an argument that could be valid in any dimension: already for cubic fourfolds one can see the large amount of cases one has to consider. For this reason, we will give the proof of the Theorem for $n\leq 4$ at the end of this section, while the case of cubic fourfolds is treated in the subsequent one, since it is more involved, even if the techniques are similar. \\
\smallskip

We stress that the implication 
$$\cH_f \mbox{ irreducible and normal} \Longrightarrow f\in \cV$$
is always true for all $n\geq 2$ as we have seen in Remark \ref{REM:Pignolotti}: the hard part of the conjecture is to prove the other implication.
\smallskip

Let us now explain the strategy that will be used for the proof of the Theorem for the various values of $n\leq 5$.

\begin{framework}
\label{strategy}
Assume that $f$ defines a smooth cubic $X=V(f)$ which is not of TS type (i.e. $[f]\in \cV$) such that $\cH_f$ is not normal. Then by Lemma \ref{LEM:triangfam} and Remark \ref{REM:Non_Normal_BigFamilyofTriangles}, we have a family $\cF$ of dimension $n-2$ of triangles for $\cH_f$ dominating, via the first projection, a component $Y_1$ of $\cD_{n-1}(f)$ of the same dimension. As we have done in the previous sections, we denote by $Y_i$ the images of the projections $\pi_i$. As just observed, $\pi_1:\cF\to Y_1$ is generically finite. 
\smallskip

If $T=([x],[y],[z])=([x_1],[x_2],[x_3])$ is a general point of $\cF$, by generic smoothness, we can assume that the differentials $d_T\pi_i:T_{\cF,T}\rightarrow T_{Y_i,[x_i]}$ are surjective; in particular, $d_T\pi_1$ is an isomorphism.
Moreover, since we are assuming $n\leq 5$, by Theorem \ref{THM:moon}, we have that none of the vertices of $T$ belongs to the cubic $V(f)$, i.e. $x_i^3\neq 0$ for $i\in \{1,2,3\}$. If we set
\begin{equation}
\label{EQ:defV1V2}
V_1=\left\langle x,y,z\right\rangle\subset A^1 \qquad \mbox{and}\qquad V_2=\Ann_{A^1}(x^2,y^2,z^2)\subset A^1,
\end{equation}
by Lemma \ref{LEM:squareindep} and by Gorenstein duality, we have that $\dim_{\bK}(V_1)=3$ and $\dim_{\bK}(V_2)=n-2$. 
Moreover, since $x^3,y^3,z^3\neq 0$, we also have $V_1\cap V_2=\{0\}$. Hence, by dimension reason, one has \begin{equation}
\label{EQ:DecoA1V1sumV2}
A^1=V_1\oplus V_2.
\end{equation}
If $\{i,j,k\}=\{1,2,3\}$, by Lemma \ref{LEM:squareindep} and Lemma \ref{LEM:TGVec_Triang_First_Order}, one has also
\begin{equation}
\dim_{\bK}\Ann_{A^1}(x_j^2,x_k^2)=n-1\qquad  \dim_{\bK}\Ann_{A^1}(x_j^2,x_k^2)/\langle x_i\rangle=n-2.
\end{equation}
By dimension reason and since $x_i\not \in V_2$ we have a canonical isomorphism $V_2\simeq \Ann_{A^1}(x_j^2,x_k^2)/\langle x_i\rangle$ induced by the inclusion $V_2\hookrightarrow \Ann_{A^1}(x_j^2,x_k^2)$ followed by the quotient by $\langle x_i\rangle$.
By Lemma \ref{LEM:TGVec_Triang_First_Order} and since we have that $d_{T}\pi_i$ surjective, we have
$$T_{Y_i,[x_i]}\subseteq \Ann_{A^1}(x_j^2,x_k^2)/\langle x_i\rangle \simeq V_2$$
so we can interpret $d_T{\pi_i}$ as maps $T_{\cT,T}\to V_2$. Being $d_T\pi_1$ an isomorphism, we have then the endomorphisms 
\begin{equation}
\psi_m=d_T\pi_m\circ (d_T\pi_1)^{-1}:V_2\to T_{Y_m,[x_m]}\hookrightarrow V_2\qquad \mbox{for } m\in \{2,3\}
\end{equation}
In the proof of Theorem \ref{THM:star}, we will start by analyzing these specific maps, ruling out both the cases where $\psi_2$ (or $\psi_3$) can be diagonalized or not and ultimately proving that a family of triangles of dimension $n-2$ can not exist.
\smallskip

As we have done in the proof of Lemma \ref{LEM:TGVec_Triang_First_Order}, to a tangent vector $\un{v}=(x',y',z')\in T_{\cT,T}$, we can associate the "first order deformation" of $T$ in the direction of $\un{v}$ (for brevity, $\un{v}$-deformation of $T$), which we write in a compact way as 
\begin{equation}
T+\un{v}=(x+tx',y+ty',z+tz').
\end{equation}
\end{framework}

\begin{lemma}
\label{LEM:Ann_Triang}
Let $\cF$ be a family of triangles. Assume furthermore that both $\cF$ and $Y_1$ have dimension $n-2$. If the general element $T=([x],[y],[z])\in \cF$ has no vertices on $X=V(f)$, then $x\cdot :V_2\to A^2$ is injective.
\end{lemma}

\begin{proof}
Since $Y_1$ has dimension $n-2$ its general point $[x]$ is in $\cD_{n-1}(f)\setminus \cD_{n-2}(f)$, i.e. the kernel of the multiplication map $x\cdot :A^1\to A^2$ has dimension $2$. The point $[x]$ is also the vertex of an element $T=([x],[y],[z])$ of $\cF$, and thus $\ker(x\cdot)=\langle y,z\rangle\subseteq V_1$. On the other hand, by the assumptions, one has $V_1\cap V_2={0}$ as observed above. 
\end{proof}

%--------------------------------------------------------------------------------------------------
%--------------------------------------------------------------------------------------------------

\bigskip
Let us now show the Theorem \ref{THM:star} in the first cases.

\begin{proposition}\label{PROP:SURFACE}
Theorem \ref{THM:star} is true for $n\in\{2,3\}$.  
\end{proposition}
\begin{proof}
By contradiction, let us assume that $[f]\in\cV$ and $\dim(\Sing(\cH_f))=n-2$. Notice that, if $n=2$, since $\Sing(\cH_f)=\cD_1(f)$, we have a contradiction by Theorem \ref{THM:charactTS}. Hence, we can assume $n=3$. \smallskip

Fix the notation explained in the framework (see \ref{strategy}). Since $n=3$, we have that 
$$\Ann_{A^1}(x^2,y^2,z^2)=V_2=\langle u\rangle$$ for suitable $u\in A^1\setminus\{0\}$. 
All projections of $\cF$ have dimension exactly $1$ by Lemma \ref{LEM:bound} so $T_{\cF,T}=\langle u(A,B,C)\rangle$ with $A,B,C\in \bK^*$. By Lemma \ref{LEM:TGVec_Triang_First_Order} the associated first order deformation $T+u(A,B,C)$ has to satisfy
$$Bux+Auy=0, \qquad Cux+Auz=0, \qquad Cuy+Buz=0.$$
We can then observe that the three independent points $Bx+Ay, \ Cx+Az, \ Cy+Bz$ belong to the kernel of the multiplication by $u$. In other words, $\iota(u)\supseteq \bP^2$, so that $[u]\in \cD_1(f)$. Hence, as before, we have a contradiction.
\end{proof}

\begin{proposition}
\label{PROP:3fold}
Theorem \ref{THM:star} is true for $n=4$: for a smooth cubic threefold $X=V(f)$, the Hessian quintic threefold $\cH_f$ is normal if and only if $f$ is not of TS type.
\end{proposition}

\begin{proof}
By contradiction, let us assume that $[f]\in\cV$ and $\dim(\Sing(\cH_f))=2$. Then we are in the situation described in \ref{strategy}: $T=([x],[y],[z])$ will denote a general triangle in $\cF$ (recall that by Theorem \ref{THM:moon} none of its vertices belongs to the cubic $X=V(f)$). 
Since $n=4$, we have that $\dim(V_2)=2$. Recall that we have the endomorphisms 
$$\psi_i=d_T\pi_i\circ (d_T\pi_1)^{-1}:V_2\simeq T_{Y_1,[x]}\to V_2\qquad i\in \{2,3\}$$
which have image $T_{Y_2,[y]}$ and $T_{Y_3,[z]}$, respectively. We have that either one of the two is diagonalizable or that none is. We treat differently the two cases.
\smallskip

{\bf{Case (I)}}: Let us suppose that at least one of the two above endomorphisms is diagonalizable. W.l.o.g. we can assume that $\{u,w\}$ is a basis of $V_2$ whose elements are eigenvectors for $\psi_2$. Then two independent tangent vectors to $\cF$ in $T$ are given as
\begin{equation}
\label{EQ:3fold_defv}
\un{v}=(u,Au,Cu+Dw) \qquad\mbox{ and }\qquad  \un{v}'=(w,Bw,Eu+Fw)
\end{equation}
for suitable $A,B,C,D,E,F$ scalars depending on the triangle.
\smallskip

We recall that, for a given tangent vector $\un{v}=(x',y',z')$ at $T=([x],[y],[z])$, we have the relation $xy'+yx'=0$ by Lemma \ref{LEM:TGVec_Triang_First_Order}. For brevity, we refer to this relation with the notation $(\un{v})_{xy}$. We denote by $(\un{v})_{xz}$ and $(\un{v})_{yz}$ the analogous relations.
\smallskip

For example, using the vectors in \eqref{EQ:3fold_defv} we have the corresponding first order deformations
$$T+t\un{v}=(x+tu,y+tAu,z+t(Cu+Dw))\qquad T+s\un{v}'=(x+sw,y+sBw,z+s(Eu+Fw))$$
which yield
\begin{equation}
\label{EQ:threefold_vvprimo}
\begin{array}{ll}
(\un{v})_{xy}: Aux+yu=0 \qquad & (\un{v}')_{xy}: Bwx+yw=0 \\
(\un{v})_{xz}: Cux+Dwx+zu=0 \qquad & (\un{v}')_{xz}: Eux+Fwx+zw=0 \\
(\un{v})_{yz}: Cuy+Dwy+Auz=0 \qquad & (\un{v}')_{yz}: Euy+Fwy+Bwz=0.
\end{array}
\end{equation}

By elementary operations, one gets two new relations:
\begin{equation*}
2ACux+D(A+B)wx=0   \qquad E(A+B)ux+2FBwx=0    
\end{equation*}

By Lemma \ref{LEM:Ann_Triang}, the multiplication map $x\cdot:V_2\to A_2$ is injective so 
\begin{equation}
\label{EQ:th1.coeff}
AC=0 \qquad D(A+B)=0 \qquad E(A+B)=0\qquad FB=0.
\end{equation}

Let us now observe that $A$ and $B$ can't be simultaneously zero when evaluated in a general triangle $T$, otherwise the second projection of $\cF$ would be zero dimensional, contradicting Lemma \ref{LEM:bound}.
\smallskip

\begin{lemma}
\label{LEM:no_eigenvalues_zero}
    In this situation, for $T$ general, none of the two eigenvalues of the endomorphism $\psi_2$ can be zero.
\end{lemma}
\begin{proof}
    Let $T=([x],[y],[z])$ be a general triangle of $\cF$. W.l.o.g. we can assume by contradiction that $A\neq 0$ and $B=0$. We claim that $Y_2\subseteq \iota(z)$ and $Y_3\subseteq \cD_2(f)$.
    \smallskip
    
    Since $A\neq0$, we also get that $C=D=E=0$ by Equation \eqref{EQ:th1.coeff}. Hence, the first order deformations of $T$ are given by 
    \begin{equation}
    \label{EQ:First_order_Lemma}
    T+t\un{v}=(x+tu,y+tAu,z)\quad \mbox{ and }\quad T+s\un{v}'=(x+sw,y,z+sFw).
    \end{equation}    
    Since the differential maps of the projections from $\cF$ are surjective, we have $\dim(Y_1)=2$ and $\dim(Y_2)=\dim(Y_3)=1$. 

    Consider the curve $C_T=\pi_3^{-1}([z])$. By construction the tangent to $C_T$ in $T$ is spanned by $\un{v}$ which is projected to $u$ and $Au$ via $d_T(\pi_1|_{C_T})$ and $d_T(\pi_2|_{C_T})$ respectively. Hence, $\pi_1(C_T)$ is a curve in $Y_1$ and $\pi_2(C_T)=Y_2$. In particular, $\pi_1(C_T)\cup Y_2\subseteq \iota([z])$, as claimed. Moreover, since $\cD_1(f)=\emptyset$ by assumption (by Theorem \ref{THM:charactTS}), we get that $\iota([z])\simeq\bP^2$ up to the case where it is a projective line coinciding both with $Y_2$ and $\pi_1(C_T)$. But in this last case, we would have an involution on $Y_2\simeq\bP^1$, which yields a fixed point and then a singular point for $V(f)$. Hence $Y_3\subseteq \cD_2(f)$ as claimed.
    \smallskip

    Notice that the same argument can be used to prove that $Y_3\subseteq \iota([y])$ and thus that $Y_2\subseteq \cD_2(f)$. By Proposition \ref{PROP:PhiInjective}, one can see that for two general points $[z]$ and $[z']$ in $Y_3$, we have that $\iota([z])\neq\iota([z'])$. Hence, $Y_2\subseteq \iota([z])\cap \iota([z'])=\bP^1$ and thus $Y_2=\bP^1\subseteq \cD_2(f)$. This is impossible by Theorem \ref{THM:charactTS} since $f$ is not of TS type.
\end{proof}

From the above Lemma \ref{LEM:no_eigenvalues_zero}, since $T$ is general, we have that both $A$ and $B$ are not zero, hence by Equation \ref{EQ:th1.coeff} we also get $B=-A$. Indeed, if $A+B\neq0$, we would obtain $C=D=E=F=0$, which is not possible by Lemma \ref{LEM:bound}.  Since $A=-B\neq 0$, from Equations \eqref{EQ:th1.coeff} we have $C=F=0$. Then, the first order deformation, in this case, can be written as
\begin{equation}
T+t\un{v}=(x+tu,y+tAu,z+tDw)\qquad T+s\un{v}'=(x+sw,y-sAw,z+sEu)    
\end{equation}
and the conditions \eqref{EQ:threefold_vvprimo} are equivalent to
\begin{equation}
\label{EQ:threef_diag}
u(y+Ax)=0 \qquad zw+Exu=0 \qquad w(y-Ax)=0 \qquad zu+Dxw=0.
\end{equation}
Moreover, we can not have $D=E=0$ as observed above, so we can assume $D\neq 0$.
\smallskip

Since these equations hold, by assumption, for the general point $T\in\cF$, by deforming $T$ at the first order in the direction of $\un{v}$, i.e. by considering a curve 
$$T(t)=([x(t)],[y(t)],[z(t)])=T+t\un{v}+t^2(\cdots),$$ 
also the corresponding eigenvectors of $\psi_2$ "move". More precisely, we have two curves
$$\gamma_u:U\to A^1\qquad \gamma_w:U\to A^1$$ defined in a neighbourhood of $0$ such that $\gamma_u(0)=u$, $\gamma_w(0)=w$ and
$\{\gamma_u(t),\gamma_w(t)\}$ is a basis of eigenvectors of $d_{T(t)}\pi_2\circ d_{T(t)}\pi_1^{-1}$. These eigenvectors satisfy equations analogous to the ones in \eqref{EQ:threefold_vvprimo} where the coefficients depend on $t$.
As observed above, the sum of the two eigenvalues of $\psi_2$ is $0$ also in a neighbourhood of $T$, so that the Equations \eqref{EQ:threef_diag} hold also locally. 
\smallskip

We can then consider an expansion of the curves 
$$\gamma_u(t)=u+tu'+t^2(\cdots)\qquad \gamma_w(t)=w+tw'+t^2(\cdots)$$
and substitute them in the Equations \eqref{EQ:threef_diag} in order to get new relations. We write $A(t)=A+A't+t^2(\cdots)$ for the curve following the eigenvalue relative to $\gamma_u(t)$ with an analogous notation for the coefficients that appear in Equations \eqref{EQ:threef_diag}.

For example, from the condition $u(y+Ax)=0$ one has
$$0\equiv \gamma_u(t)(y(t)+A(t)x(t))=(u+tu'+t^2(\cdots))
(y+Ax+t(2Au+A'x)+t^2(\cdots))$$
so we get $2Au^2+A'xu+u'(Ax+y)=0$. One can do the same reasoning for the $\un{v}'$-deformation and we also can use two "parameters" to take into account in a compact description the deformation of $T$ in the direction of $t\un{v}+s\un{v}'$. In this way, $u$ and $w$ "deform" at first order as
$$u+tu'+su''   \qquad\mbox{ and }\qquad w+tw'+sw''$$
respectively. Moreover, we can assume that $u',u''$ and $w',w''$ don't depend on $u$ and $w$, respectively. This argument yields the relations
\begin{equation}\label{EQ:th1.1}
2Au^2+A'xu+(Ax+y)u'=0\qquad u''(y+Ax)+A''xu=0
\end{equation}
\begin{equation}\label{EQ:th1.2}
w'z+Dw^2+Eu^2+E'ux+Eu'x=0\qquad 2Euw+w''z+E''ux+Eu''x=0
\end{equation}
\begin{equation}\label{EQ:th1.3}
w'(y-Ax)-A'xw=0\qquad w''(y-Ax)-A''xw-2Aw^2=0
\end{equation}
\begin{equation}\label{EQ:th1.4}
2Duw+D'xw+Dw'x+u'z=0\qquad Dw^2+Dxw''+D''xw+Eu^2+zu''=0
\end{equation}

%----------------------------------------------------
First of all, observe that multiplying by $x$ Equation $\eqref{EQ:th1.3}_I$, since $xy=0=x^2w$, we get
$$x^2w'=0.$$
Hence, multiplying by $x$ Equation $\eqref{EQ:th1.4}_I$, we obtain 
\begin{equation}
\label{EQ:th1.xuw}
xuw=0.
\end{equation}
Since $xuw=0$, from Equations \eqref{EQ:threef_diag} one can easily see that also
\begin{equation}
\label{EQ:th1.zuu}
zu^2=zw^2=yuw=0.
\end{equation}
We claim now that $A'=A''=0$ and $xu^2,xw^2\neq 0$. Indeed, let us observe that the product $xu$ vanishes if multiplied by $x, \ y, \ z, \ w$. If $xu^2=0$ too, then by the Gorenstein duality in the apolar ring $A_f$ we would get that $xu=0$, which is not possible as observed with Lemma \ref{LEM:Ann_Triang}.
% This is not possible: since $x$ lives in $Y_1$ which has dimension $2$, we would have a surface in $\cD_2(f)$, absurd by Proposition \ref{PROP:Comp_Of_Gamma}. 
In the same way, one sees that $xw^2\neq0$. Then, recalling that $u(Ax+y)=0=w(y-Ax)$, we can multiply by $u$ and by $w$ respectively Equations $\eqref{EQ:th1.1}_{II}$ and $\eqref{EQ:th1.3}_I$, getting $A''xu^2=0$ and $A'xw^2=0$, and so the claim:
\begin{equation}
\label{EQ:AAprimo}
A'=A''=0 \qquad xu^2,xw^2\neq 0.
\end{equation}

\begin{lemma}
The tangent vector $w'$ and $u''$ are trivial.
\end{lemma}
\begin{proof}
Let us prove it for $w'$. Since we can assume that $w'$ does not depend on $w$, we can write it as $w'=\alpha x+\beta y+\gamma z+\delta u$. Multiplying Equation $\eqref{EQ:th1.3}_I$ by $x$ and $y$ we get $x^2w'=0$ and $y^2w'=0$ respectively. Moreover, multiplying by $z$ Equation $\eqref{EQ:th1.2}_I$ we have also $z^2w'=0$. These last conditions yield
$$\alpha x^3=\beta y^3= \gamma z^3=0,$$
but since no vertex for the general triangle $T$ belongs to $V(f)$ we have that $\alpha=\beta=\gamma=0$. Finally, since we have just shown that $A'=0$, from Equation $\eqref{EQ:th1.3}_I$, we get 
$$\delta u(y-Ax)=0.$$
Since, from Equations \eqref{EQ:threef_diag}, we get $uy=-Aux$, we would have $-2\delta Axu=0$, which implies that $\delta=0$, by Lemma \ref{LEM:Ann_Triang}.
% Since, for $T$ general, $[x]$ cannot be in $\cD_2(f)$, we have $xu\neq 0$ and so $\delta=0$. 
Then $w'=0$ as claimed. The same reasoning can also be used to prove that $u''=0$.
\end{proof}

\begin{remark}
\label{REM:Enon0}
    From the above lemma, one can see that we can assume that also $E\neq0$. Indeed, if $E\equiv 0$ locally (and thus we can simply set $E'=E''=0$ in the above equations), from Equation $\eqref{EQ:th1.2}_I$ we would have $w^2=0$. Hence $[w]$ would be a singularity for $V(f)$, which is not possible. 
\end{remark}

We claim now that $wu^2=0$ and $w^2u=0$. Multiplying by $u$ Equation $\eqref{EQ:th1.4}_I$, one gets
\begin{equation}
\label{EQ:threefold_wu1}
    2Du^2w+zuu'=0.
\end{equation}

Since we have just shown that for $T$ general also the condition $zu^2=0$ is satisfied (see Equation \eqref{EQ:th1.zuu}), we can deform it at the first order:
\begin{equation}
\label{EQ:threefold_wu2}
0\equiv (z+tDw+sEu)(u+tu')^2 \mod \langle t,s\rangle^2 \qquad \mbox{ and so }\qquad Du^2w+2zuu'=0.
\end{equation}
Putting together Equations \eqref{EQ:threefold_wu1} and \eqref{EQ:threefold_wu2}, one gets
$$u^2w=0.$$
Let us now do the same for the second claim. Multiplying by $w$ the Equation $\eqref{EQ:th1.2}_{II}$, we get $2Euw^2+zww''=0$. Moreover, by deforming the condition $zw^2=0$ (see Equation \eqref{EQ:th1.zuu}), we get $Euw^2+2zww''=0$. As before, putting together these last conditions, one gets 
$$uw^2=0$$
by using $E\neq 0$ (see Remark \ref{REM:Enon0}).
\smallskip

We claim now that $zuw=0$. In order to show this last claim, let us deform the condition just obtained, i.e. $u^2w=0$:
$$0\equiv (u+tu')^2(w+sw'') \mod \langle t,s\rangle^2 \qquad\mbox{ and so }\qquad uwu'=0.$$
Write $u'$ as $\alpha x+\beta y+\gamma z+\delta w$ for simplicity. Since $xuw=yuw=uw^2=0$ and since $z^2w=0$ by definition of $w$, one has 
$$0=uwu'=\gamma zuw\qquad z^2u'=\gamma z^3.$$ If $\gamma\neq0$ we have done, let us then assume $\gamma=0$: we get $z^2u'=0$ and multiplying by $z$ Equation $\eqref{EQ:th1.4}_I$ we get $2Dzuw=0$ as desired.
\smallskip

Finally, having $zuw=0$ yields a contradiction: in this case, from Equations \eqref{EQ:threef_diag}, we would have $xw^2=0$, which is not possible by Equation \eqref{EQ:AAprimo}. Hence, neither $\psi_2$ nor $\psi_3$ can be diagonalizable for $T$ general.
\smallskip

%-------------------------------------------------
%-------------------------------------------------

{\bf{Case(II)}}: For $T$ general, the map $\psi_2$ is not diagonalizable. We can choose a basis $\{u,w\}$ of $V_2$ in such a way that $\psi_2$ is written in its Jordan normal form 
$$\begin{bmatrix}A&1\\0&A\end{bmatrix}.$$
As done before, the corresponding first order deformations are
$$T+t\un{v}=(x+tu,y+tAu,z+t(Cu+Dw))\qquad T+s\un{v}'=(x+sw,y+s(u+Aw),z+s(Eu+Fw)),$$
with $C,D,E,F$ not all simultaneously zero (by Lemma \ref{LEM:bound}).

These, using Lemma \ref{LEM:TGVec_Triang_First_Order}, yield
\begin{equation}
\label{EQ:threefold_vvprimo_nondiag}
\begin{array}{ll}
(\un{v})_{xy}: Axu+yu=0 \qquad & (\un{v}')_{xy}: xu+Axw+yw=0 \\
(\un{v})_{xz}: Cxu+Dxw+zu=0 \qquad & (\un{v}')_{xz}: Exu+Fxw+zw=0 \\
(\un{v})_{yz}: Cyu+Dyw+Azu=0 \qquad & (\un{v}')_{yz}: Eyu+Fyw+zu+Azw=0.
\end{array}
\end{equation}

Again, by elementary operations, one gets:
\begin{equation*}
(2AC+D)ux+(2AD)wx=0   \qquad (2AE+C+F)ux+(2AF+D)wx=0   
\end{equation*}

By Lemma \ref{LEM:Ann_Triang}, the multiplication map $x\cdot:V_2\to A_2$ is injective so 
\begin{equation}
\label{EQ:th1.coeff_nondiag1}
AD=0 \qquad 2AC+D=0 \qquad 2AF+D=0\qquad 2AE+C+F=0.
\end{equation}
Notice that in the case where $A\neq0$, from the above relations, one easily sees that also $D=C=F=E=0$, which is not possible, as we have stressed before, so we can assume 
$$A=0,\qquad D=0 \qquad \mbox{and}\qquad F=-C.$$
\smallskip
One can then observe that the matrix associated to the endomorphism $\psi_3$ with respect to the basis ${u,w}$ is of the form
$$\begin{bmatrix}C&E\\0&-C\end{bmatrix}$$
If $C\neq0$, the map $\psi_3$ would be diagonalizable: this is not possible for $T$ general as proved in Case (I). We can then assume that 
$$C=F=0 \qquad \mbox{ and }\qquad E\neq 0.$$

We are then considering the first order deformations
$$T+t\un{v}=(x+tu,y,z))\qquad T+s\un{v}'=(x+sw,y+su,z+sEu),$$
with $E\neq0$, and the relations \eqref{EQ:threefold_vvprimo_nondiag} are then equivalent to 
\begin{equation}
\label{EQ:threef_lastcase}
yu=0 \qquad \ zu=0 \qquad xu+yw=0 \qquad \ Eux+zw=0.
\end{equation}

By considering the $\un{v}$-deformation and the $\un{v}'$-deformation of the first two equations we obtain
\begin{equation}\label{EQ:th1.11}   
yu'=0 \qquad yu''+u^2=0\qquad zu'=0 \qquad zu''+Eu^2=0.
\end{equation}

We claim now that
\begin{equation}
\label{EQ:threefold_Keru2}
\Ann_{A^1}(u^2)=\langle x,y,z,u\rangle\qquad u^2w\neq 0.
\end{equation}
Conditions $yu^2=zu^2=0$ and $xu^2=0$ follow easily by multiplying by $u$ or by $x$ the equations in \eqref{EQ:threef_lastcase}. One obtains $u^3=0$ from Equation $\eqref{EQ:th1.11}_{II}$ after multiplying by $u$ and by remembering that $yu=0$. Since $u^2$ annihilates $x,y,z$ and $u$, it cannot annihilate $w$ too, since, otherwise, $[u]$ would give a singular point for $V(f)$.
\smallskip

As a consequence of the above relation notice that $([y],[z],[u])$ is a triangle for $\cH_f$ since $yz=yu=zu=0$. Hence, by Lemma \ref{LEM:squareindep}, we have $\dim(\langle y^2,z^2,u^2\rangle)=3$. Being $([x],[y],[z])$ a triangle and by Equation \eqref{EQ:threefold_Keru2} we can conclude
\begin{equation}
\label{EQ:Threefold_Anny2z2u2}
\Ann_{A^1}(y^2,z^2,u^2)=\langle x,u\rangle.
\end{equation}

We claim now that $u''\in \langle x,u\rangle$. By the above relation, it is enough to show that $y^2u''=z^2u''=u^2u''=0$. The first relation comes from $\eqref{EQ:th1.11}_{II}$ if we multiply both terms by $y$ and use \eqref{EQ:Threefold_Anny2z2u2}. One gets the second relation working on the Equation $\eqref{EQ:th1.11}_{IV}$ and using $E\neq 0$. To get the third and last relation, let us simply observe that we have shown that the equation $u^3=0$ holds for the general triangle $T$ in $\cF$ and so, we can write its $\un{v}'$-deformation:
$$0\equiv (u+su'')^3 \mod s^2 \qquad \mbox{ which yields } \qquad u^2u''=0$$
as claimed.
\smallskip

As a consequence of the last claim we can write $u''=\alpha x+\beta u$ for suitable $\alpha,\beta\in \bK$. Now consider Equation $\eqref{EQ:th1.11}_{II}$ and recall that $yu=0$ by Equation $\eqref{EQ:threef_lastcase}$. By substituting one obtains
$$0=yu''+u^2=y(\alpha x+\beta u)+u^2=u^2,$$
which is impossible since $V(f)$ is smooth. This concludes the analysis of Case (II) and, consequently, the proof of the main theorem for the case of cubic threefolds.
\end{proof}

%-----------------------------------
%-----------------------------------
%-----------------------------------
%-----------------------------------
\section{Proof of main theorem: the cubic fourfold case}
\label{SEC:star_CubicFourfoldCase}

In this section we prove Theorem A in the last remaining case: given {\it any} smooth cubic fourfold $X=V(f)$, the Hessian variety $\cH_f$ is normal and irreducible if and only if $f$ is not of TS type.
\smallskip

We set ourselves in the framework described in \ref{strategy}. We assume by contradiction that given $X=V(f)$ a smooth cubic fourfold, with $f$ which is not of TS type the associated variety $\cH_f$ is not normal. Then there exists an irreducible $3$-dimensional family $\cF$ of triangles for $\cH_f$ with the first projection dominating a $3$-dimensional component of $\Sing(\cH_f)$. Fixing a general triangle $T=([x],[y],[z])\in\cF$ for $\cH_f$, let us now study the behaviour of $\psi_2$ and $\psi_3$ as endomorphisms of $V_2$.
We distinguish the following mutually exclusive cases:
\begin{enumerate}[(a)]
    \item for the general $T$, $\psi_2$ (or $\psi_3$) has a Jordan decomposition with one Jordan block;
    \item for the general $T$, $\psi_2$ (or $\psi_3$) has a Jordan decomposition with two Jordan blocks;
    \item for the general $T$, $\psi_2$ and $\psi_3$ are diagonalizable.
\end{enumerate}
We will rule out all the possibilities, by proving the following Lemmas \ref{LEM:oneJordan} (for the case $(a)$), \ref{LEM:twoJordans} (for $(b)$), \ref{LEM:diagonalizable}, and \ref{LEM:diagonalizable:VeryBadCase} (both of them dealing with the case $(c)$). The last one, concerning a particular subcase of $(c)$, is proved in the dedicated subsection \ref{SUBSEC:lastlemma}.
\smallskip

Let us start by ruling out case $(a)$.

\begin{lemma}\label{LEM:oneJordan}
    For $T$ general, neither the map $\psi_2$ nor $\psi_3$ can have a Jordan decomposition with only one block. 
\end{lemma}

\begin{proof}
Let us suppose, w.l.o.g, that $\psi_2$ has a Jordan decomposition with one Jordan block. We can choose a basis $\{u,v,w\}$ of $V_2=\Ann_{A^1}(x^2,y^2,z^2)$ in such a way that $\psi_2$ is written in its Jordan normal form with $w$ as eigenvector. Then we have a basis $\{x,y,z,u,v,w\}$ of $A^1$ with the first three vectors such that $x^3,y^3,z^3\neq 0$. Then three independent tangent vectors to $\cF$ in $T$ are given as
\begin{equation*}
%\label{EQ:4fold_defv}
\un{v}=(u,Au+v,Bu+Cv+Dw),\quad \un{v}'=(v,Av+w,Eu+Fv+Gw),\quad  \un{v}''=(w,Aw,Hu+Iv+Lw)
\end{equation*}
for suitable scalars depending on the triangle. 
\smallskip

As done in the case of threefolds, one uses Lemma \ref{LEM:TGVec_Triang_First_Order} in order to obtain conditions from the first order deformations associated to $\un{v},\un{v}'$ and $\un{v}''$. By elementary operations between these equations and by using Lemma \ref{LEM:Ann_Triang}, one gets the following relations on the coefficients appearing in the above description of the tangent vectors:
\begin{align*}
%\label{EQ:fo1.coeff}
   2AB+E=2AC+B+F & =2AD+C+G=2AE+H=0\\
   2AF+E+I=2AG+F+L& =2AH=2AI+H=2AL+I=0    
\end{align*}

Note that if $A\neq 0$, then one has that all the other coefficients have to be $0$, which is not possible since the second and third projections cannot send $\cF$ to a point by Lemma \ref{LEM:bound}. Hence $A=0$ and then $E=H=I=C+G=B+F=F+L=0$. We can then write the above tangent vectors as
\begin{equation*}
%\label{EQ:4fold_defv}
\un{v}=(u,v,Lu-Gv+Dw),\quad \un{v}'=(v,w,-Lv+Gw),\quad  \un{v}''=(w,0,Lw)
\end{equation*}
Then, the above-mentioned equations can be reduced to the following system of equations:
\begin{equation}
\label{EQ:fourfold_vvprimo}
\begin{array}{ll}
(\un{v})_{xy}: xv+yu=0 \qquad & (\un{v})_{xz}: Lxu-Gxv+Dxw+zu=0 \\
(\un{v}')_{xy}: xw+yv=0 \qquad & (\un{v}')_{xz}: -Lxv+Gxw+zv=0 \\
(\un{v}'')_{xy}: yw=0\qquad & (\un{v}'')_{xz}: Lxw+zw=0
\end{array}
\end{equation}

As usual, if $T$ is deformed in the direction of $t\un{v}+s\un{v}'+r\un{v}''$ we have the corresponding deformation $u+tu'+su''+ru'''$ of $u$ (and analogously the ones for $v$ and $w$).
\smallskip

{\bf {Claim}}: $uw^2\neq 0$, $L=0$ and $zw=0$.\\
First of all, let us study $\ker(w^2\cdot : A^1\to A^3)$. Clearly $yw^2=0$ by Equation $(\un{v}'')_{xy}$. Moreover, if we multiply by $w$ Equation $(\un{v}')_{xy}$ and use Equation $(\un{v}'')_{xy}$, we get $xw^2=0$. Similarly, one gets $zw^2=0$ upon multiplying by $w$ Equation $(\un{v}'')_{xz}$. Since Equation $(\un{v}'')_{xy}$ holds for $T$ general, one can deform it in the direction of $t\un{v}+s\un{v}'+r\un{v}''$ and obtain
$$0=(y+tv+sw)(w+tw'+sw''+rw''') \mod (t,s,r)^2.$$
This yields
$$yw'+wv=0 \ \ \ \ \ yw''+w^2=0.$$
If we multiply by $w$ these relations, we get $w^2v=0$ and $w^3=0$. Hence we have $$\langle x,y,z,v,w\rangle\subseteq \ker(w^2\cdot: A^1\to A^3).$$ 
Observe now that $uw^2\neq 0$. Indeed, if $uw^2=0$ then we would also have $w^2\cdot A^1=\{0\}$ so, by Gorenstein duality, this would imply $w^2=0$, which contradicts the smoothness of the cubic fourfold $V(f)$.  
\smallskip    
    
Finally, by deforming Equation $(\un{v}'')_{xz}$ and multiplying by $w$, using the various vanishings obtained before, we get $2Luw^2=0$ and thus $L=0$. Then one has the claim by Equation $(\un{v}'')_{xz}$.
\smallskip 

{\bf {Claim}}: $zv=0$, $G=0$ and $D\neq 0$.\\
By deforming equations $(\un{v})_{xy}$ in the direction of $t\un{v}+s\un{v}'$ we get respectively
$$xv'+2uv+yu'=0\qquad\mbox{ and}\qquad xv''+v^2+yu''+uw=0.$$
If one multiplies these by $z$, one obtains $zuv=zv^2=0$. One can now observe that $\ker(zv\cdot:A^1\to A^3)=A^1$, so by Gorenstein duality, one has $zv=0$, as claimed. since $xw\neq 0$ (by Lemma \ref{LEM:Ann_Triang}), from Equation $(\un{v}')_{xz}$ one obtains $G=0$ and, consequently, by Lemma \ref{LEM:bound}, also $D\neq 0$.
\smallskip 

{\bf {Claim}}: $uw^2=0$.\\
Since $zv=0$ for $T$ general, we can deform this equation in the direction of $t\un{v}$. We get $zv'+Dvw=0$, and so $Dv^2w=0$, if we multiply by $v$. Being $D\neq 0$ one has also $v^2w=0$. Let us now deform $(\un{v})_{xz}$ and $(\un{v}')_{xy}$ in the direction of $t\un{v}$ in order to get $$D'xw+Dxw'+2Duw+zu'=0\quad \mbox{ and }\quad xw'+wu+yv'+v^2=0.$$
Upon multiplying by $w$ one gets $xww'+2uw^2=xww'+uw^2=0$, which yields $xww'=uw^2=0$. This is impossible as observed in the first claim above.
\end{proof}

% ---------------------------------------------
% ---------------------------------------------
Let us now prove that case $(b)$ can't be realised.

\begin{lemma}\label{LEM:twoJordans}
For $T$ general, neither the map $\psi_2$ nor $\psi_3$ can have a Jordan decomposition with two blocks.
\end{lemma}

\begin{proof}
Let us suppose, w.l.o.g, that $\psi_2$ has a Jordan decomposition with two Jordan blocks. As done in Lemma \ref{LEM:oneJordan}, we can choose a basis $\{u,v,w\}$ of $V_2=\Ann_{A^1}(x^2,y^2,z^2)$ in such a way that $\psi_2$ is written in its Jordan normal form with $v$ and $w$ as eigenvectors. Then, in this case, we can write three independent tangent vectors to $\cF$ in $T$ as
\begin{equation*}
%\label{EQ:4fold_defv}
\un{v}=(u,Au+v,Cu+Dv+Ew),\quad \un{v}'=(v,Av,Fu+Gv+Hw),\quad  \un{v}''=(w,Bw,Iu+Lv+Mw)
\end{equation*}
for suitable scalars depending on the triangle. 
\smallskip

As done in the previous cases, one gets the following relations involving the coefficients appearing in the above description of the tangent vectors:
\begin{align*}
    2AC+F=2AD+C+G=& AF=2AG+F=BM=0 \\
    E(A+B)+H=H(A+B)= &I(A+B)=L(A+B)+I=0  
\end{align*}
We distinguish four cases, depending on the vanishing of the two eigenvalues.
\smallskip

{\bf Case (I): $A=B=0$}.\\
Since $A=B=0$, we also have $F=H=I=C+G=0$. Among the various equations obtained by deforming at first order the general triangle $T$ one gets
\begin{equation}
\label{EQ:Fourfold_Triang_twoJordan_Case1}
yv=0\qquad yw=0.
\end{equation}

We claim now that $uv^2\neq 0$. By deforming at first order the equation $yv=0$ in the direction of $t\un{v}$, one gets
$yv'+v^2=0$ which implies that
$$\langle x,y,z,v,w\rangle\subseteq\ker(v^2).$$
On the other hand, this has to be an equality, otherwise we would have $v^2=0$ by Gorenstein duality. In particular, $uv^2\neq 0$.
\smallskip

From the Equation $yv'+v^2=0$, one can also see that $yv'\neq0$, otherwise we would contradict the smoothness of $V(f)$. We claim now that $yv'=0$ so we conclude Case (I). 
\smallskip

Since $T$ is a triangle and by Equations \eqref{EQ:Fourfold_Triang_twoJordan_Case1}, we have  $\langle x,z,v,w\rangle\subseteq \ker(y\cdot :A^1\to A^2)$, so, in order to prove $yv'=0$, it is enough to show that $v'$ does not depend on $y$ and $u$. One easily sees that $0=y(yv'+v^2)=y^2v'$. Since, by assumption $\langle x,z,u,v,w\rangle=\ker(y^2\cdot :A^1\to A^3)$, we get that $v'$ does not depend on $y$. Moreover, since we have just shown that for $T$ general also the equation $v^3=0$ holds, we can deform it and, in the same way, one proves that $v'$ does not depend on $u$. 
\smallskip

{\bf Case (II): $A=0, B\neq 0$}.\\
In this case we also have $E=F=C+G=H=I=L=M=0$ so that
\begin{equation*}
\un{v}=(u,v,Cu+Dv),\quad \un{v}'=(v,0,-Cv),\quad  \un{v}''=(w,Bw,0)
\end{equation*}
Among the equations deduced by deforming the general triangle one gets the conditions
\begin{equation}
\label{EQ:Fourfold_Triang_twoJordan_Case2}
yv=0\qquad zw=0\qquad xv+yu=0\qquad Bxw+yW=0\qquad -Cxv+zv=0\qquad Cxu+Dxv+zu
\end{equation}

We claim now that $C=0$.  First of all notice that
$$\langle x,y,z,v,w\rangle\subseteq\ker(xv).$$
Indeed, we have $v\in V_2$ by assumption so $x^2v$ and since $T$ is a triangle we also have $xyv=xzv=0$. The last two vanishing can be easily obtained form Equations $\eqref{EQ:Fourfold_Triang_twoJordan_Case2}$ upon a multiplication by $v$ and $w$ (and by recalling that we are assuming $B\neq 0$). In particular, we have $xuv\neq 0$, by Lemma \eqref{LEM:Ann_Triang}. 

Then by taking the fifth and sixth equations in $\eqref{EQ:Fourfold_Triang_twoJordan_Case2}$ multiplied by $u$ and $v$ respectively, one has 
$$-Cxuv+zuv=Cxuv+zuv=0$$
so $Cxuv=0$. Then, since $xuv\neq 0$, we have necessarily $C=0$ for the general triangle and so $D\neq0$ by Lemma \ref{LEM:bound}. In particular we have $\psi_3(u)=Dv\neq 0$ and $\psi_3(v)=\psi_3(w)=0$. Hence, for the general triangle $T$, $\psi_3$ is not diagonalizable and has two Jordan blocks with eigenvalues both equal to $0$. This is impossible, as seen in Case (I).
\smallskip

{\bf Case (III): $A\neq 0, B=0$}.\\
This case can be treated in a "geometric" way as done in Lemma \ref{LEM:no_eigenvalues_zero}. Indeed, since $A\neq 0$ and $B=0$, we have also all the other variables, besides $M$, are zero. Moreover, by Lemma \ref{LEM:bound}, $M\neq 0$. In particular, the tangent vectors to $\cF$ in $T$ is spanned by
\begin{equation*}
%\label{EQ:4fold_defv}
\un{v}=(u,Au+v,0),\quad \un{v}'=(v,Av,0),\quad  \un{v}''=(w,0,Mw)
\end{equation*}
and the varieties $\pi_i(\cF)=Y_i$ have dimension $3,2$ and $1$, respectively. 
As done in the other cases, by studying the relations coming from the deformation at the first order, one can easily see that $Y_3$ is contained in $\cD_2(f)$ and $Y_2$ is a surface living in $\cD_3(f)$. Since $Y_2$ cannot be contained in $\cD_2(f)$ (otherwise we would have singular points for $V(f)$ by Proposition \ref{PROP:Comp_Of_Gamma}), for the general $[y]\in Y_2$ we have that $\iota([y])\simeq\bP^2$.
\smallskip

For a general triangle $T=([x],[y],[z])$, consider the curve $C_T=\pi_2^{-1}([y])$. The tangent to $C_T$ in $T$ is generated by $\un{v}''$ which is projected to $w$ and $Mw$ via $d_T(\pi_1|_{C_T})$ and $d_T(\pi_3|_{C_T})$ respectively. Hence, $\pi_1(C_T)$ is a curve in $Y_1$ and $\pi_3(C_T)=Y_3$. 

In particular, $\pi_1(C_T)\cup Y_3\subseteq \iota([y])\simeq\bP^2$. By varying the point $[y]$, the kernel has to move, since the curve $\pi_1(C_T)$ has to cover the threefold $Y_1$. Hence, $Y_3$ lies in the intersection of distinct projective planes: we have $\bP^1\simeq Y_3$ and thus a line in $\cD_2(f)$. This implies, by Theorem \ref{THM:charactTS}, that $f$ is of TS type, against our assumptions.
\smallskip

{\bf Case (IV): $A,B\neq0$}.\\
First of all, notice that assuming $A,B\neq 0$, implies that $A+B=0$. Indeed, if we assume also $A+B\neq 0$, we would obtain that all the other coefficients are equal to $0$. This is impossible by Lemma \ref{LEM:bound}.
Then
\begin{equation*}
%\label{EQ:4fold_defv}
\un{v}=(u,Au+v,Ew),\quad \un{v}'=(v,Av,0),\quad  \un{v}''=(w,Bw,Lv)
\end{equation*}
with $E,L$ not both zero. In particular, $\psi_3$ is not diagonalizable for the general triangle $T$ and all its eigenvalues are zero. This is impossible as seen in the previous cases.
\end{proof}

% ---------------------------------------------
% ---------------------------------------------

As a consequence of Lemmas \ref{LEM:oneJordan} and \ref{LEM:twoJordans}, the maps $\psi_2$ and $\psi_3$ have to be diagonalizable. In what follows, we rule out this remaining case, splitting it up into two lemmas, the second of which is postponed in the following subsection.

% ---------------------------------------------

\begin{lemma}\label{LEM:diagonalizable}
    For $T$ general, neither the map $\psi_2$ nor $\psi_3$ can be diagonalizable.
\end{lemma}

\begin{proof}
As a consequence of Lemma \ref{LEM:oneJordan} and \ref{LEM:twoJordans}, we have that $\psi_2$ and $\psi_3$ are both diagonalizable for the general triangle. We can choose a basis $\{u,v,w\}$ of $V_2=\Ann_{A^1}(x^2,y^2,z^2)$ in such a way that $\psi_2$ is in diagonal form. Thus three independent tangent vectors to $\cF$ in $T$ are
\begin{equation*}
%\label{EQ:4fold_defv}
\un{v}=(u,Au,Du+Ev+Fw),\quad \un{v}'=(v,Bv,Gu+Hv+Iw),\quad  \un{v}''=(w,Cw,Lu+Mv+Nw)
\end{equation*}
for suitable coefficients depending on the triangle. 
\smallskip

As done so far, one gets the following relations:
\begin{align*}
    AD=BH=CN=0 \qquad & E(A+B)=G(A+B)=0 \\
    F(A+C)=L(A+C)=0 \qquad & I(B+C)=M(B+C)=0.
\end{align*}
Notice that $A,B$ and $C$ can not be all equal to zero by Lemma \ref{LEM:bound}. Hence, we distinguish three cases, depending on the vanishing of the three eigenvalues.
\smallskip

{\bf Case (I): $A=B=0$ and $C\neq 0$}.\\
In this case, one can easily see that three tangent vectors to $\cF$ at $T$ can be written as
\begin{equation*}
\un{v}=(u,0u,Du+Ev),\quad \un{v}'=(v,0,Gu+Hv),\quad  \un{v}''=(w,Cw,0)
\end{equation*}
for suitable coefficients so that $\dim(Y_1)=3$, $\dim(Y_2)=1$ and $\dim(Y_3)\in\{1,2\}$. 
\smallskip

\begin{itemize}

\item {\bf Claim:} $\dim(Y_3)=1$ and $Y_2\subseteq\cD_2(f)$.\\
Among the equations obtained by deforming the general triangle $T$, one has $yu=yv=0$. Moreover, by definition of triangle, we clearly have also $yx=yz=0$: the general $[y]\in Y_2$ lives in $\cD_2(f)$, i.e. for $[y]\in Y_2$ general $\iota([y])\simeq\bP^3$. In the same way, since one gets $zw=0$, one also has $Y_3\subseteq\cD_3(f)$. Notice that the surface $\pi_2^{-1}([y])=S_T$ projects onto a surface in $Y_1$ via $\pi_1$ and dominates $Y_3$ via $\pi_3$. This means that $Y_3\subset\iota([y])\simeq\bP^3$. Since these kernels can not be fixed with $[y]$ varying, one gets that $Y_3\subseteq\iota([y_1]\cap\iota([y_2])\simeq\bP^s$ with $s\in\{1,2\}$. From this, one can see that $\dim(Y_3)=1$: indeed, if $Y_3$ is a surface then we necessarily have $s=2$ and $Y_3\simeq\bP^2$, but this means there exists a $2$-projective plane in $\cD_3(f)$, which is impossible by Theorem \ref{THM:charactTS}
\smallskip

    \item {\bf Claim:} $Y_3\subseteq \cD_2(f)$.\\
Since $Y_3$ is a curve, the endomorphism $\psi_3|_{\langle u, v\rangle}$ has necessarily rank $1$, i.e. there exists a vector $au+bv$ that is sent to $0$ by $\psi_3$. Among the first order conditions given by the tangent vectors above, one has
$$(Du+Ev)x+zu=(Gu+Hv)x+zv=0.$$
Then, since $\psi_3(au+bv)=a(Du+Ev)+b(Gu+Hv)=0$, one has $z(au+bv)=0$. Hence $\iota([z])\simeq \bP^3$ and we have that $Y_3$ is contained in $\cD_2(f)$ too.
\smallskip

Being $f$ not of TS type, and being $Y_3$ a curve in $\cD_2(f)$, we have that $Y_3$ is not a line. On the other hand $Y_3\subseteq \iota([y_1]\cap\iota([y_2])\simeq \bP^2$ so, the general triangle $T$ is such that $\iota([y])\simeq \bP^3$ contains a fixed $\bP^2$, which we denote by $\Pi$. To conclude, let us take a general point $[\eta]\in\Pi$: by symmetry the general $[y]\in Y_2$ is such that $[y]\in\iota([\eta])$, and so $Y_2\subset\iota([\eta])\simeq\bP^r$. Since $Y_2\not\simeq\bP^1$, we have that $r\geq2$: this means that $\Pi\simeq \bP^2\subseteq \cD_3(f)$, which yields a contradiction as above.
\end{itemize}

\noindent Let us stress that having a $\bP^2$ contained in $\cD_3(f)$ is a phenomenon that happen exactly when $V(f)$ is a smooth, non-cyclic cubic of $TS$ type as described in the specific Example \ref{EX:cub+cub}.
\smallskip

{\bf Case (II): $A=0$ and $B,C\neq0$}.\\
This case can not occur. It will be treated in Lemma \ref{LEM:diagonalizable:VeryBadCase}. 
\smallskip

% --------------------------------------------------------------

{\bf Case (III): $A,B,C\neq0$}.\\
Being $A,B,C\neq 0$ one has $D=H=N=0$. Notice that the three values $A+B,A+C$ and $B+C$ can not be simultaneously zero, moreover at least one of them has to be $0$, since otherwise we would get $D=E=F=G=H=I=L=M=N=0$ which is impossible by Lemma \ref{LEM:bound}. W.l.o.g, we distinguish $2$ cases: either $A+C=0$ and $A+B,B+C\neq 0$ or $A+B=A+C=0$ and $B+C\neq 0$.
In the first case, $\psi_3$ is diagonalizable with one zero eigenvalue, whereas in the second case one has the same conclusion or that $\psi_3$ is not diagonalizable and its Jordan normal form has $1$ Jordan block with $0$ as the only eigenvalue. Both conclusions yield a contradiction as observed in the previous cases or in Lemma \ref{LEM:oneJordan}.
\end{proof}

% ---------------------------------------------
% ---------------------------------------------

\subsection{The end of the proof}
\label{SUBSEC:lastlemma}

To end the proof of Theorem \ref{THM:star}, we have to rule out a last remaining possibility which could arise in the case where the both the maps $\psi_2$ and $\psi_3$ are diagonalizable (case $(c)$, as stated at the beginning of Section \ref{SEC:star_CubicFourfoldCase}). This subsection is devoted to this subcase, which is ruled out with the following Lemma \ref{LEM:diagonalizable:VeryBadCase}, which yields also the end of the proof of the main theorem. To prove this last Lemma, we start by analyzing the usual framework \ref{strategy} and obtaining different relation that both the vertices of the general triangle and the tangent vectors to it have to satisfy. After that, we will use these conditions to reconstruct the cubic fourfolds which the framework is, in this case, associated with, showing that these do not actually satisfy the hypotheses we are setting.

\begin{lemma}\label{LEM:diagonalizable:VeryBadCase}
    For $T$ general, neither the map $\psi_2$ nor $\psi_3$ can be diagonalizable with dimension of the kernel equal to $1$.
\end{lemma}

\begin{proof}
We refer to the notations introduced at the beginning of Lemma \ref{LEM:diagonalizable}. W.l.o.g, we can set $A=0$ so that $B,C\neq 0$ by hypothesis. Then, one has $E=F=G=H=L=N=0$. First of all, notice that $B+C=0$. Indeed, otherwise, we would get $I=M=0$ so $v$ and $w$ would be two eigenvectors for $\psi_3$ with associated eigenvalue $0$. This is impossible as observed in Case (I).

Then, the tangent vectors to $\cF$ at $T$ can be written as
\begin{equation*}
\un{v}=(u,0,Du),\quad \un{v}'=(v,Bv,Iw),\quad  \un{v}''=(w,-Bw,Mv)
\end{equation*}
with $B\neq 0$ and $(D,I,M)\neq (0,0,0)$ by Lemma \ref{LEM:bound}. Moreover, notice that $I,M\neq 0$ since, otherwise, we would have that $\psi_3$ is not diagonalizable (and this can not happen for $T$ general by Lemma \ref{LEM:twoJordans}).
\smallskip

The first order conditions obtained as a consequence of Lemma \ref{LEM:TGVec_Triang_First_Order} are

\begin{equation}
\label{EQ:fourfold_diag_cattivo}
\begin{array}{ll}
(\un{v})_{xy}: yu=0 \qquad & (\un{v})_{xz}: (Dx+z)u=0 \\
(\un{v}')_{xy}: (Bx+y)v=0 \qquad & (\un{v}')_{xz}: Ixw+zv=0 \\
(\un{v}'')_{xy}: (-Bx+y)w=0\qquad & (\un{v}'')_{xz}: Mxv+zw=0
\end{array}
\end{equation}

Consider the following subsets of $\cD=\bK[x,y,z,u,v,w]$
\begin{small}
$$
\cM^{nv}=\{x^3,y^3,z^3,xu^2,xv^2,xw^2,yv^2,yw^2,zvw,uv^2\}
$$
$$M_0=\{xy,xz,yz\}\cup \left(\{x^2,y^2,z^2\}\cdot \{u,v,w\}\right)\qquad M_1=\{yu,xuv,xuw,zuv,zuw\}$$
$$M_2=\{u^2v,u^2w\}\qquad M_3=\{xvw,yvw,zv^2,zw^2\}\qquad  M_4=\{v^3,w^3\}\qquad M_5=\{v^2w,vw^2\}$$
\end{small}
and the elements
\begin{small}
$$r_1=x(Iw^2-Mv^2)\qquad r_2=u(Mv^2-Iw^2)\quad \mbox{ and }\quad r_3=u(Dw^2+Mvw).$$
\end{small}
Notice that all the monomials in $M_0$ are $0$ in $A_f=\cD/\Ann_\cD(f)$ by the conditions imposed by our framework. We want to prove that the same holds for all the elements in $M_i$ for $i\in \{1,\dots, 5\}$ and for $r_1,r_2$ and $r_3$, whereas all the monomials in $\cM^{nv}$ are {\it not } $0$ (in $A_f$).

If $T$ is deformed in the direction of $t\un{v}+s\un{v}'+r\un{v}''$ the corresponding first order deformation of $u$ is written as $u+tu'+su''+ru'''$ (and analogously the ones for $v,w, B,D, I$ and $M$).
\smallskip

% -----------------------------------------------------------------

{\bf Claim:} All the monomials in $M_1$ and in $M_2$ are $0$. \\
One has $yu=0$ from Equation $(\un{v})_{xy}$. Upon multiplying by $u$ the other equations in \eqref{EQ:fourfold_diag_cattivo}, one gets the vanishing for the monomials in $M_1$. By deforming Equation $(\un{v})_{xy}$ in the direction of $s\un{v}'+r\un{v}''$, we get
$$yu''+Buv=yu'''-Buw=0.$$
Multiplying by $u$ these relations and by using the vanishing $yu=0$, one gets the claim.
\smallskip

% -----------------------------------------------------------------

{\bf Claim:} All the monomials in $M_3$ are $0$.\\
Observe that it is enough to show hat $xvw=0$: all the other vanishings come from Equations \eqref{EQ:fourfold_diag_cattivo} after multiplication by $v$ or $w$ and vanishing in $M_1$ or $M_2$.\\
Let us deform Equation $(\un{v}')_{xy}$ in the direction of $r\un{v}''$: 
\begin{equation}
\label{EQ:fourfolds_(3)r}
B'''xv+Bxv'''+yv'''=0.
\end{equation}
Multiplying by $x$ we get the relation $x^2v'''=0.$ Recalling that $V_2=\langle u,v,w\rangle=\Ann_{A_1}(x^2,y^2,z^2)$, let us consider the vanishing $x^2v=0$: its deformation in the direction of $r\un{v}''$ yields $x^2v'''+2xvw=0$. Since, as just shown, $x^2v'''=0$, we get $xvw=0$, as claimed.
\smallskip

% -----------------------------------------------------------------

{\bf Claim:} All the monomials in $\cM^{nv}$ besides $uv^2$ are not $0$ and $B''=B'''=0$.\\
We have $x^3,y^3,z^3\neq 0$ by assumption, since the general triangle of $\cF$ can not have a vertex on the cubic fourfold. As a consequence of the framework and since the monomials in $M_1$ and $M_3$ are $0$, we get that $\langle x,y,z,u,w\rangle\subseteq\ker(xv\cdot: A^1\rightarrow A^3)$. Hence, $xv^2$ can not be zero, otherwise, by Gorenstein duality, also $xv$ would be $0$, which is not possible as observed in Lemma \ref{LEM:Ann_Triang}. In the same way, one gets that also $xu^2$ and $xw^2$ are not $0$. For the remaining monomials, they have to be different from $0$, otherwise one would get a contradiction multiplying the equations in \eqref{EQ:fourfold_diag_cattivo} by $u, v$ or $w$.
\smallskip
For the second claim, observe that multiplying by $v$ Equation \eqref{EQ:fourfolds_(3)r}, since $(Bx+y)v=0$ by Equation $(\un{v}')_{xy}$, one gets $B'''xv^2=0$. Being $xv^2\neq0$, as just shown, we have also that $B'''=0$, as claimed. In order to show that $B''=0$, one proceeds in an analogous way by deforming Equation $(\un{v}'')_{xy}$ in the direction of $s\un{v}'$: 
\begin{equation}
\label{EQ:fourfolds_(5)s}
-B''xw-Bxw''+yw''=0.
\end{equation}
One gets the claim by multiplying by $w$.
\smallskip

% -----------------------------------------------------------------

{\bf Claim:} The monomials in $M_4$ are $0$.\\
Let us consider the first order deformation of $(\un{v}')_{xy}$ and $(\un{v}'')_{xy}$ in the direction of $s\un{v}'$ and $r\un{v}''$ respectively:
\begin{equation}
\label{EQ:fourfolds_(3)s+(5)r}
B''xv+Bxv''+2Bv^2+yv''=0\qquad -B'''xw-Bxw'''-2Bw^2+yw'''=0.
\end{equation}
Since $B''=B'''=0$ as shown in the previous claim, if one multiplies the above Equations \eqref{EQ:fourfolds_(3)s+(5)r} by $v$ and $w$ respectively, one gets the claim.
\smallskip

% -----------------------------------------------------------------

{\bf Claim:} One has $w''=v'''=0$ as tangent vectors.\\
Let us start by proving that $w''=0$. Since, by construction, we have that $w\in V_2$, one can deform in the direction of $s\un{v}'$ the relations $x^2w=y^2w=z^2w=0$. Recalling that $xvw=yvw=zw^2=0$ by the previous Claims, one obtains that $w''\in V_2$, so we can write $w''=\alpha u+\beta v+\gamma w$.
By substituting this expression in Equation \eqref{EQ:fourfolds_(5)s} one has
$$-\alpha Bxu+\beta(-Bx+y)v=0,$$
which, if multiplied by $u$, gives $\alpha B xu^2=0$. Since $Bxu^2\neq 0$, one has $\alpha=0$. 
\smallskip

Being $\alpha=0$, it follows $\beta(-Bx+y)v=0$ from the above equation. On the other hand, from one has $(Bx+y)v=0$ (see Equation $(\un{v}'_{xy})$) so $\beta=0$. Indeed, otherwise we would get $xv=0$, which is not possible by Lemma \ref{LEM:Ann_Triang}. This means that $w''=0$ in $A^1/{\langle w\rangle}$.
\smallskip

In order to get $v'''=0$, one proceeds in a similar way: first of all one proves $v'''\in V_2$ starting from $v\in V_2$ and by using previous vanishings. Then, by substituting in Equation \eqref{EQ:fourfolds_(3)r} and by using Equation $(\un{v}')_{xy}$, one concludes as above.
\smallskip

% -----------------------------------------------------------------

{\bf Claim:} The monomials in $M_5$ are $0$.\\
We have shown that $zv^2=0$ and $xvw=0$ for the general triangle $T$, so we can deform these equations in the direction of $s\un{v}'$. By using Equation $(\un{v})_{xz}$ and $w''=0$, one can write these relations as
$$Iv^2w+2vzv''=I(v^2w-2xwv'')=0\qquad v^2w+xwv''=0.$$
As observed above, $I$ is not $0$, thus we deduce $v^2w=0$.
\smallskip

For the vanishing $vw^2=0$, one works in a similar way by deforming $zw^2=0$ and $xvw=0$ in the direction of $r\un{v}''$, and by using Equation $(\un{v}'')_{xz}, v'''=0$ and $M\neq 0$.

% -----------------------------------------------------------------

{\bf Claim:} $v'\in \langle v,w\rangle$ and $u'$ does not depend on $y$.\\
Proceeding as we have done above for proving $v''',w''\in V_2$, one can also obtain that $v'\in V_2=\langle u,v,w\rangle$. Consider the first order deformation of the Equation $(\un{v}')_{xy}$ in the direction of $t\un{v}$, namely
$$B'xv+Bxv'+Buv+yv'=0.$$
Upon multiplication by $u$, using $B\neq 0$ and the various vanishing shown above, one gets $xuv'=0$. Since $xu^2\neq 0$ and $xu\cdot \langle x,y,z,v,w\rangle=0$, one has that $v'\in \langle v,w\rangle$.
\smallskip

For the second claim, by deforming $(\un{v})_{xy}$ in the direction of $t\un{v}$, one gets $yu'=0$ so $y^2u'=0$ and this implies that $u'$ does not depend on $y$, since $y^3\neq 0$ and $y^2\cdot\langle x,z,u,v,w\rangle=0$.
\smallskip

% -----------------------------------------------------------------

{\bf Claim:} $uv^2\neq 0$ in $A_f$.\\
Assume, by contradiction, that $uv^2=0$. We claim that $v'=0$ as tangent vector. Consider the first order deformation of the Equation $zv^2=0$ in the direction of $t\un{v}$, i.e.
$$0=Duv^2+2zvv'=2zvv'$$
Since $v'\in \langle v,w\rangle$ (by the previous claim), $zv^2=0$ and $zvw\neq 0$, one has that $v'=0$ as tangent vector.
\smallskip

Since we are assuming that $uv^2=0$ for the general triangle in $\cF$, we can deform this equation in the direction of $t\un{v}$. This operation yields the relation $0=v^2u'+2uvv'=v^2u'$.
Now recall that $u'$ does not depend on $y$, $v^2\cdot \langle z,v,w\rangle$ by previous vanishings and $uv^2=0$, by assumption. Since $xv^2\neq 0$, from $v^2u'=0$ one has that $u'$ does not depend on $x$.
\smallskip

This yield a contradiction by deforming $x^2u=0$ in the direction of $t\un{v}$. Indeed, one has
$0=x^2u'+2xu^2=2xu^2$ but $xu^2\neq 0$.
\smallskip

% -----------------------------------------------------------------

{\bf Claim:} Elements $r_1, r_2$ and $r_3$ are $0$.\\
The relation $r_1=x(Iw^2-Mv^2)=0$ is easily obtained from Equations $(\un{v}')_{xz}$ and $(\un{v}'')_{xz}$ upon multiplication by $w$ and $v$ respectively.

We prove now that $r_2=u(Mv^2-Iw^2)=0$. Consider the first order deformation in the direction of $s\un{v}'$ of the Equations $zuw=0$ and $xuv=0$, together with Equation $\eqref{EQ:fourfolds_(3)s+(5)r}_I$ multiplied by $u$, namely
\begin{equation}
\label{EQ:fourfold_RelationsPartial}
Iuw^2+zuw''+zwu''=0\quad uv^2+xuv''+xvu''=0\quad Bxuv''+2Buv^2+yuv''+B''xuv=0.
\end{equation}
Now, since $w''=zw+Mxv=0$, and $B''=yu=0\neq B$ we have
$$Iuw^2-Mxvu''=0\qquad xuv''+2uv^2=0$$
which give the desired relation, if substituted into Equation $\eqref{EQ:fourfold_RelationsPartial}_{II}$.
\smallskip

The other relation, namely $r_3=u(Dw^2+Mvw)=0$, is obtained in a similar way from the first order deformation in the direction of $r\un{v}''$ of the Equations $(\un{v})_{xy}$ and $(\un{v})_{xz}$ upon multiplication by suitable elements (more precisely, the first one by $w$ and $Mv$ and the second one by $w$, respectively).
\smallskip

To sum up, we have proved that if we define
$$
\cR=\{r_1,r_2,r_3\}\cup \left(\bigcup_{i=0}^5 M_i\right)\cup\{\mbox{LHS of relations in }\eqref{EQ:fourfold_diag_cattivo}\}
$$
then 
\begin{equation}
\label{EQ:fourfold_cond_lin_system}
\cR \subseteq \Ann_{\cD}(f) \qquad \mbox{ and }\qquad  \cM^{nv}\cap \Ann_{\cD}(f)=\emptyset.
\end{equation}

Now we would like to partially reconstruct the cubic fourfold $f$ from the information about its apolar ring $A_f$ obtained so far. For simplicity, we are using the same symbols for the indeterminates in $S=\bK[x_0,\dots,x_5]$ and in $\cD=\bK[y_0,\dots, y_5]=\bK[x,y,z,u,v,w]$. Consider the following cubics in $S^3$: 
$$s_0=x^3\quad s_1=y^3\quad s_2=z^3\quad s_3=(x-Dz)u^2 \quad s_6=u^3$$ 
$$s_4=x(Iv^2+Mw^2)+yB(Mw^2-Iv^2)-2IMzvw\quad s_5=-2Duvw+u(Iv^2+Mw^2).$$

It is easy to see that 
$$W=\langle s_i\rangle_{i=0}^5=\{f \in S^3\,|\, \cR\subseteq \Ann_D(f)\}.$$ 
This can be checked directly by hand by writing $f=\sum\alpha_m\cdot m$ where $m$ runs over the set of monomials of degree $3$ in $S$. Each element in $\cR$ is a linear differential equation satisfied by $f$ and thus gives a linear closed condition on the vector space $S^3$. For example, since $r_2=u(Mv^2-Iw^2)\in \cR$, we have the corresponding condition  $2M\alpha_{uv^2}-2I\alpha_{uw^2}=0$ on the coefficients of $f$.
\smallskip

Hence, {\it any} cubic polynomial that we are analysing in this case, can be written as $f=\sum_{i=0}^{6}p_i s_i$ for suitable $p_i\in \bK$.
Having proved that $\cM^{nv}\cap \Ann_{\cD}(f)=\emptyset$ gives non-trivial open conditions: indeed, it is translated into 
\begin{equation}
\label{EQ:fourfold_nonvanishing}
p_0,p_1,p_2,p_3,p_4,p_5\neq 0
\end{equation}
so all the cubic fourfolds satisfying Conditions \eqref{EQ:fourfold_cond_lin_system} live in a dense open subset of $|W|$. Notice that the base locus of $|W|$ is the line $L=V(x,y,z,u)$. Moreover, as $p_4,B,I,M\neq 0$ and since
$$y_0(f)|_{L}=p_4(Iv^2+Mw^2)\qquad  y_1(f)|_{L}=p_4B(Mw^2-Iv^2),$$ 
we have that the general cubic in $|W|$ is indeed smooth on the points of $L$ and thus smooth everywhere by Bertini.
\smallskip

% -----------------------------------------------------------------

{\bf Claim:} One has $D\neq 0$.\\
Consider the first order deformation of $v^2w$ and of $z^2v$ in the direction of $s\un{v}'$, namely
\begin{equation}
\label{EQ:fourfold_Dnon0}
2vw v''+w''v^2=2vw v''=0\qquad z^2v''+2Izvw=0,
\end{equation}
where we also used $w''=0$.
Assume, by contradiction, that $D=0$. Then, from the expression of $f$ and as $p_5\neq 0$, we have that $uvw=0$ in $A_f$. Hence $vw\cdot \langle x,y,u,v,w\rangle=0$. Then, by previous vanishings and since $zvw\neq 0$, from Equation $\eqref{EQ:fourfold_Dnon0}_I$, we get that $v''$ does not depend on $z$. This implies that $z^2v''=0$ so Equation $\eqref{EQ:fourfold_Dnon0}_{II}$ yields $2Izvw=0$, and thus $I=0$, which is impossible.
\smallskip

% -----------------------------------------------------------------

{\bf Claim:} If $f$ satisfies the Conditions in \eqref{EQ:fourfold_cond_lin_system}, then $\Sing(\cH_f)$ is of dimension $2$ near $[x]$.\\

By changing coordinates, we can simplify a little the expression of $f$. Indeed, as $B$, $D$, $I$, $M$, $p_4$ and $p_5$ are not $0$, by an easy change of coordinates, and by redefining the $p_i$s one can write

% As $\bK$ is algebraically closed, one can consider $d_I,d_M,d_4,d_5\in \bK\setminus\{0\}$ such that $d_I^2=I,d_M^2=M,d_4^2=p_4$ and $d_5^2=p_5$. Set $\lambda=ab/D$. The change of coordinates
% $$\alpha_1([x:y:z:u:v:w])=[x:y/B:z/D:u\cdot \mu:v/(a\cdot d_4):w/(b\cdot d_4)]$$
% is such that
% $$s_0(\alpha_1(\un{x}))=x^3\qquad s_1(\alpha_1(\un{x}))=y^3/B^3\qquad s_2(\alpha_1(\un{x}))=z^3/D^3\qquad s_6(\alpha_1(\un{x}))=\mu^3\cdot u^3$$
% $$s_3(\alpha_1(\un{x}))=\mu^2\cdot (x-z)u^2\qquad \sigma_5(\alpha_1(\un{x}))=\mu p_4^{-1}\cdot \left(u(w^2+v^2)-2\lambda^{-1} uvw\right)$$
% $$s_4(\alpha_1(\un{x}))=p_4^{-1}\cdot \left(x(w^2+v^2)+y(w^2-v^2)-2\lambda zvw\right).$$
% Then, after the change of coordinates and by redefining the $p_i$s in the obvious way, all the cubics of our interest can be written as
\begin{multline}
2f=p_0 \left(x^3\right)+p_1 \left(y^3\right)+p_2 \left(z^3\right)+p_3 \left((x-z)u^2\right)+p_6 \left(u^3\right)+ \\
+ (x+u)(w^2+v^2)+y(w^2-v^2)-2\left(\lambda z+\lambda^{-1}u\right)vw
\end{multline}
with $\lambda,p_0,p_1,p_2,p_3\neq 0$.
\smallskip

By construction, $[x]\in \bP(A^1)\leftrightsquigarrow (1:0:0:0:0:0)\in \bP^n$ is a vertex of a triangle for $\cH_f$ so $[x]\in \Sing(\cH_f)$. We are assuming also that there exists a family of dimension $5-2=3$ whose general element is a triangle dominating via the first projection a component of dimension $3$ of $\Sing(\cH_f)$. Then, in order to conclude the proof of the lemma, it is enough to show that the local dimension of $\Sing(\cH_f)$ near $[x]$ is actually $2$. 
\smallskip

The Hessian matrix of $f$ is
\begin{equation}
\label{EQ:fourfold_HessianMAt}
H_f=\begin{bmatrix}
3 p_0 x & 0 & 0 & p_3 u & v & w \\
0 & 3 p_1 y & 0 & 0 & - v & w \\
0 & 0 & 3 p_2 z & -p_3 u & -\lambda w & -\lambda v\\
p_3 u & 0 & -p_3 u & p_3 x - p_3 z + 3 p_6 u & v - \lambda^{-1} w & -\lambda^{-1} v + w \\
v & -v & -\lambda w & v-\lambda^{-1} w & x - y + u & -\lambda z - \lambda^{-1} u\\
w & w & -\lambda v & -\lambda^{-1} v + w & -\lambda z - \lambda^{-1} u & x + y + u
\end{bmatrix}
\end{equation}

Since $f$ is smooth, we have that $\Sing(\cH_f)=\cD_4(f)$. In particular, $\Sing(\cH_f)$ is cut out by $21$ quintic equations corresponding to the minors of order $5$ of the Hessian matrix (there are $36$ minors but $15$ appear twice since $H_f$ is symmetric). Let $m_{ij}$ be the minor obtained by removing the $i$-th row and the $j$-th column. 
We are interested in the local expression of $\Sing(\cH_f)$ near $[x]$. Notice that the polynomial $(y_iy_j)(f)$ depends on $x$ if and only if $(i,j)\in \{(0,0),(3,3),(4,4),(5,5)\}$ so no term of $h_f$ can have as exponent of $x$ an integer greater than $4$: this is a confirmation of the fact that $[x]\in \Sing(\cH_f)$. By differentiating $m_{ij}$ it is easy to see that $[x]$ is singular for $V(m_{ij})$ if $(i,j)\not\in \{(1,1),(1,2),(2,1),(2,2)\}$. 
Consider the variety $Z=V(m_{11},m_{12},m_{22})$ and notice that $\Sing(\cH_f)\subseteq Z$ by construction.
% In particular, the Zariski tangent space of $\cD_4(f)$ at $[x]$ is the same as of $Z$. A more accurate analysis shows also that $[x]$ is singular for $V(m_{12})$ too, whereas $V(m_{11})$ and $V(m_{22})$ are smooth at $[x]$ with different tangent space: the Zarisky tangent space of $\cD_4(f)$ (and of $Z$) at $[x]$ is $3$-dimensional. 
Being defined by $3$ equations, one has that $\dim(Z)\geq 2$. We claim that $Z$ has dimension $2$ near $[x]$; to do that we will show that $(1:0:0:0:0:0)$ is isolated in $Z\cap V(u,v)=V(m_{11},m_{12},m_{22},u,v)$.
\smallskip

We compute now the local expression of $m_{11},m_{12}$ and $m_{22}$ modulo $(u,v)$ in the local ring $A_m$ where $A=\bK[y,z,u,v,w]$ and $m$ is the maximal ideal of the origin in $\bA^5$. By the explicit expression of $H_f$ in Equation \eqref{EQ:fourfold_HessianMAt}, one can easily see that
\begin{equation}
\label{EQ:fourfold_m12}
m_{12}(1,y,z,0,0,w)=-3p_0\cdot w^2\cdot(w^2-p_3\lambda^2 z(1-z))=0
\end{equation}
so one between $w$ and $w^2-p_3\lambda^2 z(1-z)$ is zero. 
\smallskip

Assume first that $w=0$.
By a direct computation, one can see that
$$m_{11}(1,y,z,0,0,0)=9p_0p_2p_3\cdot z(z-1)(y^2+\lambda^2 z^2-1)\sim z$$
$$m_{22}(1,y,z,0,0,0)=9p_0p_1p_3\cdot y(z-1)(y^2+\lambda^2 z^2-1)\sim y$$
since $p_0,p_1,p_3\neq 0$ by assumption and since both $z-1$ and $y^2+\lambda^2 z^2-1$ are invertible in $A_m$. This shows that $[x]$ is isolated in $Z\cap V(u,v,w)$.
\smallskip

Assume now that $w^2=p_3\lambda^2 z(1-z)$. One can show that $w$ appears only with even powers in $m_{ij}(1,y,z,0,0,w)$ for $i,j\in\{1,2\}$ so one can substitute $p_3\lambda^2 z(1-z)$ to $w^2$ in order to obtain the two expressions

$$
r_{11}=p_3\cdot z(z-1)\cdot
    \left(3p_2 y + (3p_2 - p_3\lambda^4)(z-1)\right)\cdot \left(3p_0 y + (p_3\lambda^2)z^2-\lambda^2(3p_0+p_3)z+3p_0\right)
$$
$$
r_{22}=9p_0p_1p_3\cdot(z-1)\cdot (y+z-1)\cdot
    \left(y^2 + \frac{p_3\lambda^2}{3p_0} yz^2 - \lambda^2\frac{3p_0 + p_3}{3p_0} yz + y +
    \frac{p_3\lambda^2}{3p_1} z(z-1)\right).
$$
In $A_m$, one has
$$
r_{11}\sim z\cdot \left(3p_2 y + (3p_2 - p_3\lambda^4)(z-1)\right)\qquad r_{22}\sim
    y\cdot g(y,z) +
    \frac{p_3\lambda^2}{3p_1} z(z-1)
$$
with $g(0,0)\neq 0$. Since $w^2=p_3\lambda^2 z(1-z)$, if we assume $z=0$ we also have that $w=0$ so we can conclude by the previous case. We can then suppose that $3p_2 y + (3p_2 - p_3\lambda^4)(z-1)=0$ in the local ring. This can happen if and only if $3p_2=\lambda^4p_3$ and $y=0$. On the other hand, if $y=0$, from the expression of $r_{11}$ one has that $z(z-1)=0$ and thus again $w^2=0$. This shows that $[x]$ is isolated in $Z\cap V(u,v)$ too and thus that the local dimension of $\cD_4(f)=\Sing(\cH_f)$ near $[x]$ is $2$.   
\end{proof}

With this lemma, we also conclude the proof of Theorem \ref{THM:star} in the case of cubic fourfolds.

% ---------------------------------------------------------
% ---------------------------------------------------------
% Bibliografia
% ---------------------------------------------------------
% ---------------------------------------------------------

\bibliographystyle{amsalpha}

\end{document}